\documentclass[reqno, 12pt, reqno]{amsart}

\usepackage{amsfonts, amsthm, amsmath, amssymb}
\usepackage{hyperref}
\hypersetup{colorlinks=false}

\usepackage[margin=1.5in]{geometry}

\RequirePackage{mathrsfs} \let\mathcal\mathscr

\numberwithin{equation}{section}

\newtheorem{theorem}{Theorem}[section]
\newtheorem{lemma}[theorem]{Lemma}

\newtheorem{prop}[theorem]{Proposition}

\theoremstyle{remark}
\newtheorem*{ack}{Acknowledgements}
\newtheorem{remark}[theorem]{Remark}

\theoremstyle{definition}
\newtheorem{defi}[theorem]{Definition}

\renewcommand{\d}{\mathrm{d}}
\renewcommand{\phi}{\varphi}
\renewcommand{\rho}{\varrho}

\newcommand{\0}{\mathbf{0}}

\newcommand{\PP}{\mathbb{P}}
\renewcommand{\AA}{\mathbb{A}}

\newcommand{\FF}{\mathbb{F}}
\newcommand{\ZZ}{\mathbb{Z}}

\newcommand{\NN}{\mathbb{N}}

\newcommand{\RR}{\mathbb{R}}

\newcommand{\TT}{\mathbb{T}}

\renewcommand{\leq}{\leqslant}

\renewcommand{\geq}{\geqslant}

\renewcommand{\bar}{\overline}

\newcommand{\h}{\mathbf{h}}

\newcommand{\g}{\mathbf{g}}

\newcommand{\x}{\mathbf{x}}

\renewcommand{\u}{\mathbf{u}}

\newcommand{\ve}{\varepsilon}

\DeclareMathOperator{\rank}{rank}

\DeclareMathOperator{\meas}{meas}

\usepackage[usenames,dvipsnames]{color}

\def \bfg { {\bf g}}
\def \bfh {{\bf h}}

\begin{document}

\title[Free rational curves and the circle method]{Free rational curves on low degree hypersurfaces and the circle method}

\author{Tim Browning}

\address{IST Austria\\
Am Campus 1\\
3400 Klosterneuburg\\
Austria}
\email{tdb@ist.ac.at}

\author{Will Sawin}
\address{Columbia University\\ 
Department of Mathematics\\
2990 Broadway\\ New York\\ NY 10027\\ USA}
\email{sawin@math.columbia.edu}


\begin{abstract}
We use a function field version of the Hardy--Littlewood circle method to 
study the locus of free rational curves  
on an arbitrary smooth projective 
hypersurface of sufficiently low degree.  On the one hand this  allows us to 
bound the dimension of the singular locus of the moduli space of rational curves on such hypersurfaces and, on the other hand, it sheds light on Peyre's reformulation of the Batyrev--Manin conjecture in terms of slopes with respect to the tangent bundle. 
\end{abstract}

\date{\today}

\maketitle

\thispagestyle{empty}
\setcounter{tocdepth}{1}
\tableofcontents

\section{Introduction}

Let $X\subset \PP^{n-1}$ be a smooth hypersurface of degree $d\geq 3$, over a field $K$
whose characteristic exceeds $d$ if it is positive.  
This paper has two aspects. On the one hand, motivated by questions in algebraic geometry,  we shall be interested in the locus  of points corresponding to free rational curves inside the 
  moduli space   $\mathcal M_{0,0}(X,e)$
of degree $e$ rational curves on $X$.  On the other hand, by working over a finite field, we shall 
establish 
a function field analogue of a recent conjecture  due to Peyre \cite{peyre-freedom} about 
the distribution of ``sufficiently free'' rational points of bounded height on Fano varieties.

\subsection{Geometry}
 
The expected dimension 
of $\mathcal M_{0,0}(X,e)$   is $(n-d)e+n-5$, a fact that is known to hold 
for generic $X$ if $n\geq d+3$, 
thanks to Riedl and Yang \cite{RY}. 
It follows from work of Browning and Vishe \cite{BV'} that
 $\mathcal M_{0,0}(X,e)$  
 is irreducible and 
 has the expected dimension 
for any smooth $X$, provided that  
$n> (5d-4)2^{d-1}.$ 
Our first result strengthens this.

\begin{theorem}\label{t:BV} 
Let $d\geq 3$, let $e\geq 1$ and let  $n>(2d-1) 2^{d-1}$. Then 
$\mathcal{M}_{0,0}(X,e)$ 
is an irreducible locally complete intersection of the expected dimension. 
\end{theorem}

We can also bound  the dimension of the singular locus of $\mathcal M_{0,0}(X,e)$, 
as follows.

\begin{theorem}\label{t:starr} 
Let $d\geq 3$, let $e \geq 1 $ and let
$
n> 3 (d-1)  2^{d-1}.$ 
Then the 
space $\mathcal{M}_{0,0}(X,e)$ is smooth outside a set of codimension at least
$$  \left(\frac{n}{2^{d-2}}-6(d-1)\right) \left( 1 + \left \lfloor \frac{e+ 1}{d-1} \right\rfloor \right) .$$
In particular, whenever these inequalities are satisfied, it is generically smooth and reduced.
\end{theorem}

For $n\geq 2d+1$ and  generic $X$ of degree $d\geq 3$,   Harris, Roth and Starr  \cite{HRS} have also  shown that  $\mathcal M_{0,0}(X,e)$ is generically smooth. 
Note that, provided  $n> 3 (d-1) 2^{d-1}$, the codimension goes to $\infty$ 
in Theorem \ref{t:starr}
when either $e$ or $n$ does, with $d$ fixed. Moreover, when both $e$ and $n$ are large with respect to $d$, the codimension is at least approximately $\frac{1}{ 2^{d-2} (d-1)}$ of the total dimension.  

Our work 
addresses some questions of Eisenbud and Harris \cite[\S 6.8.1]{3264}
concerning the Fano variety of lines  $F_1(X)=\mathcal M_{0,0}(X,1)$ associated to 
a smooth hypersurface
$X\subset \PP^{n-1}$  of degree $d$.
Specifically, their question (a) asks whether  $F_1(X)$ is reduced and irreducible if $n> d+1$ and
(b) asks  whether the dimension of the singular locus of $F_1(X)$ can be bounded in terms of $d$ alone. 
Theorems \ref{t:BV} and \ref{t:starr} answer the first question  affirmatively for $n>3(d-1)2^{d-1}$ and 
give some weak evidence in support of the second question, by showing that it grows with $n$  more slowly than the dimension of the whole space. Furthermore, we handle the analogous conjectures with higher degree curves, with no loss in the dependence on $n$, meaning that for large enough
 $e$ we do better than their predicted bound $d \leq n/e$. 
 
By comparison, Starr \cite{Starr} has  proved that if $n\geq d+e$ and $X$ is generic, then  $\mathcal M_{0,0}(X,e)$ has canonical singularities, which implies in particular that it is smooth outside a set of codimension at least $2$. 
It does not seem possible that our method will prove that $\mathcal M_{0,0}(X,e)$ has canonical singularities. By 
Mustat\v{a} \cite{Mustata} and Lang-Weil \cite{LW} this  is 
equivalent to the conjunction of an infinite sequence of Diophantine estimates (in the spirit of Definition \ref{def:Nrho}),  
but for fixed $n,d$ and $e$ it seems unlikely that the circle method is able to handle 
more than finitely many of them. 
In unpublished work, Starr and Tian use a bend-and-break approach to produce a less restrictive lower bound for the codimension of the singular locus for a general hypersurface $X\subset \PP^{n-1}$ 
of degree $d$. However, their method never proves a lower bound for the codimension greater than $n$, whereas our work achieves this if $e$ is sufficiently large. 
 
 Comparing the various results, we see that Theorem \ref{t:starr}  holds for a much more restricted range of $n$ (unless $e$ is very large relative to $d$) but it is valid for an arbitrary smooth hypersurface, rather than just a general one. 
 
 It should be possible to adapt our strategy to prove results about moduli spaces of genus $g$ curves on $X$. However, the codimension we obtain for the whole moduli space will not be any better than the codimension we can prove for the space of maps from a fixed genus $g$ curve to $X$. 
 In particular the codimension  will shrink  as  $g$ grows, 
 so the bound obtained would only be suitable for $e$ sufficiently large with respect to $g$. 

Let $\mathcal T_X$ be the tangent bundle associated to the smooth hypersurface $X\subset \PP^{n-1}$ (as defined in \cite[p.~180]{hart}, for example). 
Our remaining result deals specifically with free curves and so we recall the definition here.

\begin{defi} \label{def-free}
Let $c: \mathbb P^1 \to X$ be a rational curve and let $\rho\in \ZZ$. We 
 say that $c$ is $\rho$-free if $c^* \mathcal T_X \otimes \mathcal O_{\mathbb P^1}(-\rho)$ is globally generated.
\end{defi} 

We shall follow common convention and say  that $c$ is free if it is $0$-free, and very free if it is $1$-free. One easily checks that this  agrees with the standard definition that $c$ is free if $c^*\mathcal T_X$ is globally generated and very free if $c^* \mathcal T_X$ is ample.
The definition of free curves goes back to pioneering work of Koll\'ar--Miyaoka--Mori \cite[\S~1]{KMM'} on rational connectedness for Fano varieties, and they feature heavily  in work of 
Koll\'ar \cite[\S~II.3]{kollar}. 
We have taken Definition \ref{def-free} from work of Debarre \cite[Def.~4.5]{Debarre}, 
which appears to be the first time that the notion  of being $\rho$-free occurs, for varying $\rho\in \ZZ$.
 
\begin{remark}\label{rem:1.4}
If $c$ is a $\rho$-free  rational curve on $X$ then it follows from Definition~\ref{def-free} that 
$\deg(c^*\mathcal{T}_X)\geq \rank (c^* \mathcal{T}_X) \rho$. 
In general, the pull-back of the tangent bundle has rank $n-2$ and degree $e (n-d)$. In this way we see that no degree $e$ rational curve on $X$ is ever $(\lfloor \frac{ e(n-d)}{n-2} \rfloor +1)$-free. If $d\geq 2$ then this implies that $\rho\leq e$, for any $\rho$-free rational curve $\PP^1\to X$.
 \end{remark}

We let $U_\rho\subset 
\mathcal{M}_{0,0}(X,e)$ 
be the Zariski open set  
that parameterises  degree $e$ maps from $\mathbb P^1$ to $X$ that are  $\rho$-free. We write 
  $Z_\rho=\mathcal{M}_{0,0}(X,e)\setminus U_\rho$ for the complement.
This is the closed set parameterising 
degree $e$ maps $\mathbb P^1\to X$ that are not $\rho$-free.
We shall prove the following bound for its dimension.
 
\begin{theorem}\label{t:1}
Let $d\geq 3$ and $n>3(d-1)2^{d-1}$.
Assume that $\rho\geq -1$  and 
\begin{equation}
\label{eq:e-condition}
e\geq  (\rho+1) \left(2 + \frac{1}{d-2} \right ).
\end{equation}
Then 
\begin{equation}
\begin{split}
\label{eq:main-bound} \dim Z_\rho\leq~& 
(n-d)e+n-5  
+ 2(d-1)\left\lfloor \frac{\rho+1}{2}\right\rfloor
\\&-\left(\frac{n}{2^{d-2}}-6(d-1)\right)
\left(1+
\left\lfloor\frac{e-\rho}{d-1}\right\rfloor
- \left\lfloor \frac{\rho+1}{2}\right\rfloor\right). 
\end{split}
\end{equation}
\end{theorem}

The notion of free rational curves was originally introduced as a tool to study uniruled and rational connectedness properties of varieties. Taking $\rho=1$  it follows from Theorems \ref{t:BV} and \ref{t:1} that $U_1\neq \emptyset$ if $K$ is algebraically closed and $e$ is sufficiently large.  Hence, by appealing to  \cite[Cor.~4.17]{Debarre}, we deduce that any smooth hypersurface $X\subset  \PP^{n-1}$ of degree $d$ is 
rationally connected if $d\geq 3$ and 
$n>  3(d-1) 2^{d-1} $. 
This recovers a weak form of 
the well-known result, independently  
due 
to Campana \cite{campana} and 
Koll\'ar--Miyaoka--Mori \cite{KMM} 
that Fano varieties  are rationally connected. 
In fact  both proofs use reduction to characteristic $p$, but they use different properties of characteristic $p$ varieties, with   \cite{KMM} relying on  Frobenius pull-back and our work using  the Lang--Weil estimates.

\medskip

Theorem \ref{t:starr} is derived from  
Theorem  \ref{t:1}, which is proved using analytic number theory and builds on an approach 
employed by Browning and Vishe \cite{BV'}.  (Theorem \ref{t:BV} uses essentially the same approach as \cite{BV'}, with one improvement to a key lemma.) 
One  begins by working over a finite field $K=\FF_q$ of characteristic $>d$. We bound the dimension of $Z_\rho$ by counting the  number of  points defined over a finite extension of $\FF_q$
that lie in it.  In \S \ref{geometry-diophantine}, we will give an explicit description of this locus in terms 
of a system of two  Diophantine equations defined over the function 
field $\mathbb F_q(T)$.  Let $f\in \FF_q[x_1,\dots,x_n]$ be a non-singular form of 
degree $d$ that defines the hypersurface $X\subset \PP^{n-1}$. 
Given $\rho\in \ZZ$, we shall see that the  primary counting function of interest to us, denoted 
$N_\rho(q,e,f)$, is  the one that counts vectors  $(\g,\h)\in \FF_q[T]^{2n}$, where
$g_1,\dots,g_n$ have degree at most $e$ and no common zero, 
 with at least  one of degree exactly $e$, 
 and where $h_1,\dots,h_n$ have degree at most $e-1 - \rho$, such that 
\begin{equation}
\label{eq:system}
f(g_1,\dots,g_n)=0
\quad \text{
 and }\quad 
 \sum_{i=1}^n h_i \frac{\partial f}{\partial x_i}(g_1,\dots,g_n)=0.
 \end{equation}
 Since each partial derivative of $f$ is a degree $d-1$ polynomial,
 we obtain a linear equation for $\h\in \FF_q[T]^n$ 
where the coefficients have degree at most $(d-1)e$ in $T$.
Standard heuristics lead us to expect that, for typical $\g$, the number of available 
$\h$ is $q^{(e-\rho)(n-1)-(d-1)e}=q^{e(n-d)-\rho(n-1)}$. (In fact, we shall see in Lemmas \ref{dimension-free} and \ref{polynomial-expression} that this is true  only if the map $\PP^1\to X$ represented by 
$\g$ is $\rho$-free.) 
Thus we expect that $N_\rho(q,e,f)$ is approximated by  $q^{e(n-d)-\rho(n-1)}N(q,e,f)$, where 
 $N(q,e,f)$ is the number of vectors $\g\in \FF_q[T]^n$ such that $f(\g)=0$,  where
$g_1,\dots,g_n$ have degree at most $e$ and no common zero, 
 with at least one of degree exactly $e$.

In \S \ref{diophantine-exponential}, we apply the function field version of the Hardy--Littlewood circle method to study the system of degree $d$   equations \eqref{eq:system}, expressing the number of solutions as an integral of an exponential sum. We shall show that the major arc contribution to this integral cancels almost exactly with the expected approximation
$q^{e(n-d)-\rho(n-1)}N(q,e,f)$. 
In \S \ref{s:minor}, we prove an upper bound on all other arcs, taking special care to make 
all of our implied constants depend explicitly  on the size of the finite field $q$.
The standard way of proceeding  involves $d-1$ applications of Weyl differencing, a process
that  would ultimately 
 require 
$n>3(d-1)2^{d}$  variables overall. We shall gain a 50\% reduction in the number of  variables
 by exploiting the special shape of the Diophantine system \eqref{eq:system}.
Finally,  we  bring everything together and apply the Lang--Weil estimates \cite{LW} to turn the bound for  $\#Z_\rho(\FF_q)$ into a bound for the dimension of $Z_\rho$. An application of    spreading-out shows that the dimension bound holds over an arbitrary base field $K$ such that $\mathrm{char}(K)>d$ if it is positive.

\subsection{Arithmetic}

In our  geometric investigation of $Z_\rho$ we take the point of view that $e$ and $\rho$ are fixed and $q\to \infty$. In this subsection we assume that the finite field is fixed, but we allow  the parameters $e$ and $\rho$ to tend to infinity appropriately.

Suppose  that $V$ is a  smooth projective geometrically integral Fano variety 
defined over a number field $K$. 
For suitable Zariski open subsets $U\subset V$ the Batyrev--Manin conjecture \cite{FMT89}
makes a precise prediction about the asymptotic behaviour of the counting function 
$$
N_U(B)=\#\{x\in U(K): H_{\omega_V^{-1}}(x)\leq B\},$$ 
as $B\to \infty$, where $H_{\omega_V^{-1}}:V(K)\to \RR$ is an anticanonical height function. These conjectures are flawed, however, since  it has been discovered 
that the presence of Zariski dense thin sets in $V(K)$ may skew the expected asymptotics. Recently, Peyre \cite{peyre-freedom}  has embarked on an ambitious programme to repair the conjecture by associating a measure of ``freeness'' $\ell(x)\in [0,1]$ to any $x\in V(K)$ and only counting those rational points for which $\ell(x)\geq \ve_B$, where $\ve_B$ is a function of $B$ decreasing to zero sufficiently slowly. (See \cite[Def.~6.11]{peyre-freedom} for a precise statement.) Peyre's function $\ell(x)$ is defined using 
Arakelov geometry and 
the  theory of slopes associated to the tangent bundle $\mathcal{T}_V$. 

We can lend support to Peyre's freedom prediction \cite[\S 6]{peyre-freedom} by studying smooth hypersurfaces of low degree in the setting of global fields of positive characteristic. Let 
$X\subset \PP^{n-1}$ be a smooth hypersurface of degree $d$ defined over a finite field $\FF_q$ whose characteristic exceeds $d.$  We put 
\begin{equation}\label{eq:NX}
N_X^{\ve\text{-free}}(B)=\#\left\{x\in X(K): \ell(x)\geq \ve, ~H_{-\omega_X}(x)\leq q^B\right\},
\end{equation}
where $K=\FF_q(T)$ is the rational function field and 
 $\ell(x)$ will be defined in 
\S \ref{s:peyre}. The expectation is that for a suitable range of $\ve$, 
$N_X^{\ve\text{-free}}(B)$ should have the same asymptotic behaviour as the usual counting function 
$N_X(B)$, as $B\to \infty$. 
The following result 
confirms this and will be proved in  \S \ref{s:peyre}.

\begin{theorem}\label{t:P}
Let $d\geq 3$, let $n>3(d-1)2^{d-1}$ and let
$$
0\leq \ve< \frac{n-1}{(n-d)(d-1)^22^{d-1}}.
$$
Then  there exists $\delta>0$ such that 
$$
N_X^{\ve\text{-free}}(B)=c_X
 q^{B} +O\left(q^{(1-\delta)B} \right),
$$
as $B\to \infty$, 
where $c_X$ is the  
function field analogue of the  
constant predicted by Peyre \cite{peyre-duke}.
Furthermore, 
the implied constant only depends on $q$ and $f$.
\end{theorem}

Note that this result does not require $\ve_B$ to decrease to zero, but only to stay below some fixed constant. This may be because the hypersurface $X$ has Picard rank one, since Peyre 
has shown in  \cite[\S 7.2]{peyre-freedom} that for the product 
 $\mathbb P^1 \times \mathbb P^1$ one requires  $\ve_B\to 0$ for the asymptotic formula to be true.
Finally, one can see from the arguments in Theorem \ref{t:P} that we can take the upper bound for $\ve$ to be significantly greater than $\frac{n-1}{(n-d)(d-1)^22^{d-1}}$ when $n$ is large. 
(In fact, the cutoff is allowed to approach $\frac{1}{d+1}$ as $n\to \infty$.)

With appropriate adjustments to the proof of Theorem \ref{t:P}, it is also possible to handle the
corresponding result for smooth hypersurfaces of low degree defined over $\mathbb{Q}$, with Poisson summation taking the place of the Riemann--Roch arguments that feature in \S \ref{geometry-diophantine}.   This is the object of our concurrent work~\cite{sawin3}.

\begin{ack}
The authors are grateful to 
Paul Nelson, Per Salberger and Jason Starr
for useful  comments. 
While working on this paper the first  author was
supported by EPRSC 
grant \texttt{EP/P026710/1}. The research was partially conducted during the period the second author served as a Clay Research Fellow, and partially conducted during the period he  was supported by Dr. Max R\"{o}ssler, the Walter Haefner Foundation and the ETH Zurich Foundation.
\end{ack}

 \section{Examples}

As usual,  $X\subset \PP^{n-1}$ is assumed to  be a smooth hypersurface of degree $d\geq 3$, over a field $K$ whose characteristic is either $0$ or $>d$.
While the latter condition  arises very naturally in our argument (as explained in Remark \ref{:P}), 
the following result shows that 
the statement of Theorem \ref{t:1} is actually false when it is dropped. 

\begin{lemma}
Let $K=\bar \FF_p$ for a prime $p$  and let $X\subset \PP^{n-1}$ be the 
Fermat hypersurface 
$$
x_1^{d}+\dots+x_n^{d}=0.
$$
Assume that $p\nmid d$ and $d\neq a p^r-1$ 
for any $r\in \NN$ and $a\in\{0,\dots, p-1\}$.
Then $X$ is smooth, none of the curves in $ \mathcal{M}_{0,0}(X,1)$ are $(-1)$-free, and 
$\dim \mathcal{M}_{0,0}(X,1)>2n-d-5$.
\end{lemma}

\begin{proof}
The moduli space of $n$-tuples of polynomials of degree $\leq 1$ satisfying the equation $
x_1^{d}+\dots+x_n^{d}=0
$ is a $\mathrm{GL}_2$-bundle over the moduli stack $\mathcal M_{0,0}(X,1)$ parameterising lines in $X$, because for each line we can choose any basis of the corresponding two-dimensional vector space. Thus  its dimension is equal to $4+ \dim \mathcal M_{0,0}(X,1)$. This space is is cut out by $d+1$ equations in $2n$ variables, where $\binom{d}{i}$ divides all coefficients of the $i$th equation, for $0\leq i\leq d$. 
By Lucas' theorem it follows that $p\mid \binom{d}{i}$ if and only if at least one of the base $p$ digits of $i$ is greater than the corresponding base $p$ digit of $d$. In this way we see that $p\mid \binom{d}{i}$ for some $0\leq i\leq d$ if and only if $d$ does not take the form $ap^r-1$ for some 
$a\in \{0,\dots,p-1\}$.
But then the space is cut out by fewer than $d+1$ equations in $2n$ variables. This implies that it has dimension greater than $2n-d-1$, whence $\dim \mathcal{M}_{0,0}(X,1)>2n-d-5$. 
Furthermore, 
since the dimension near each curve is greater than the expected dimension, it follows from Lemma \ref{free-examples} that they are not $(-1)$-free.   Finally, 
the Fermat hypersurface is smooth over $K$ if and only if  $p\nmid d$.
\end{proof}

This example generalises a discussion of Debarre \cite[\S 2.15]{Debarre}. It shows 
 that for typical $p<d$ the 
 statements of Theorems \ref{t:BV} and \ref{t:1} are false for fields of characteristic $p$.

Returning to the general setting, the following result provides examples  
of curves that are not $\rho$-free.

\begin{lemma}
\label{lem:ex1}	
 Let $d,m,n\in \NN$ with $d \geq 3$ and $ m \leq n/2$. Let $K$ be an infinite field. There exists a non-singular form  $f(x_1,\dots,x_n)$ over $K$ of degree $d$, such that 
$$
f(x_1,\dots,x_m,0,\dots,0)=\frac{\partial f}{\partial x_j} (x_1,\dots,x_m,0\dots,0)=0$$ 
for all $x_1,\dots,x_m$ and all $j \leq n-m$. 
For such a polynomial, every map $c:\mathbb P^1 \to  X$ of degree $e$ that factors through $\mathbb P^{m-1} \subseteq X \subseteq \mathbb P^{n-1}$ fails to be $(\lfloor   \frac{e (m-d)}{m-1} \rfloor +1)$-free. The moduli 
space of such rational curves  has dimension $m(e+1)-4$. 
\end{lemma}

Let $X\subset \PP^{n-1}$ be a smooth hypersurface with underlying polynomial $f$, as in the lemma. 
Taking $m=d$ and $\rho=0$, we see that when $n>2d$ 
the space $Z_1$ of non-very-free rational curves $\PP^1\to X$ of degree $e$ has dimension 
at least $
d(e+1)-4.
$

\begin{proof}[Proof of Lemma \ref{lem:ex1}] 
Without the non-singularity condition, 
the space of such polynomials is linear. The singular polynomials form a closed subset. To prove the existence, it is sufficient to show that this subset has codimension $1$. The set of singular polynomials is the projection from the product of this linear space with $\mathbb P^{n-1}$ of the set of pairs of a point and a polynomial singular at that point.  For elements in $\mathbb P^{m-1} \subseteq \mathbb P^{n-1}$, the space of polynomials singular at that point has codimension $m$, as it is defined by the
$m$ 
independent conditions  $\frac{\partial f}{\partial x_j} (x_1,\dots,x_m,0\dots,0)=0$ for 
$n-m+1 \leq j\leq n$.  
For all other elements, we claim that the $n$ conditions $\frac{\partial f}{\partial x_j}(x_1,\dots,x_n)=0 $ for $1\leq j\leq n$  define a codimension $n$ subspace. 
To see this we may take a linear form 
$l$ 
in the last $n-m$ coordinates that is nonzero at that point.
Then  the $n$-dimensional space of polynomials generated by 
$x_j l^{d-1} $  
for $1\leq j\leq n$ lie in the linear subspace, since  $d-1 \geq 2$. But only the zero element in that subspace satisfies all $n$  conditions.
It follows that the singular locus is the union of the projection of a codimension 
$m$ 
bundle on  $\mathbb P^{m-1}$ and a codimension $n$ bundle on its complement in $\mathbb P^{n-1}$. Thus the singular locus has  codimension at least one, as desired.

For the freeness, we use the Euler exact sequence 
\begin{equation}\label{eq:euler}
0 \to \mathcal O_{\mathbb P^{n-1}} \to \mathcal O_{\mathbb P^{n-1}}(1)^n \to \mathcal T_{\mathbb P^{n-1}}\to 0.
\end{equation}
 Consider the map $\mathcal O_{\mathbb P^{n-1}}(1)^n  \to \mathcal O_{\mathbb P^{n-1}}(1)^m$ given by projection onto the last $m$ factors. Because $m \leq n/2$ the composition of this projection with the map $\mathcal O_{\mathbb P^{n-1}} \to \mathcal O_{\mathbb P^{n-1}}(1)^n $ vanishes on $\mathbb P^{m-1}$. So over $\mathbb P^{m-1}$, we obtain a map $\mathcal T_{\mathbb P^{n-1}} \to 
 \mathcal O_{\mathbb P^{n-1}}(1)^m$.

Next consider the exact sequence $0 \to \mathcal T_{\mathbb P^{n-1}} \to \mathcal T_X \to \mathcal O_{X} (d) \to 0$ on $X$. The second map of this sequence is the dot product with the derivative of $f$. By assumption on $f$, restricted to $\mathbb P^{m-1}$, this map factors through the projection onto the last $m$ vectors. Hence we obtain an exact sequence 
$$
0 \to 
\mathcal{V} 
 \to \mathcal O_{\mathbb P^{m-1} }(1)^m \to \mathcal O_{\mathbb P^{m-1}} (d) \to 0
$$ 
whose kernel $\mathcal{V}$ is a vector bundle on $\mathbb P^{m-1}$ of degree $m-d$, which arises as a quotient of $\mathcal T_{X}$.

For $c : \mathbb P^1 \to X$ a map of degree $e$ whose image lies in $\mathbb P^{m-1}$,  $c^* \mathcal{V}$ is a vector bundle of degree $e(m-d)$ on $\mathbb P^1$ which arises as a quotient of $c^* \mathcal T_X$. Because $c^* \mathcal{V}$ splits as a direct sum of $m-1$ line bundles, it must contain some line bundle summand of degree at most $\frac{ e(m-d)}{m-1}$, and we can round down to the nearest integer. Hence $c^* \mathcal T_X$ has some line bundle summand of degree at most $\lfloor \frac{ e(m-d)}{m-1} \rfloor$ and hence $c$ is not  $(\lfloor \frac{ e(m-d)}{m-1} \rfloor+1)$-free.

The dimension estimate is the standard calculation for the moduli space of rational curves in projective space.
\end{proof}

Even for a general hypersurface there are some non-very-free curves. Indeed, for such a variety, the moduli space of lines has dimension $2n-d-5$, and each line admits a $(2e+1)$-dimensional moduli space of degree $e$ maps from $\mathbb P^1$ to that line. Because the pull-back of the tangent bundle to a line has rank $n-2$ and degree $n-d$, it contains some summand of degree at most $0$ as soon as $d \geq 2$, and so every pull-back of it has a summand of the same degree, and so these degree $e$ coverings of lines fail to be $1$-free. Hence, for a general hypersurface
$X\subset \PP^{n-1}$ of degree $d,$ we have 
$\dim Z_1\geq  2(n+e)-d-7$.

These examples show that the dimension of the moduli space of non-very-free curves can grow linearly in $n$ and it can grow linearly in $e$. We do not know if it can grow linearly in $ne$, as the dimension of 
$\mathcal{M}_{0,0}(X,e)$ does.

\section{Vector bundles on $\mathbb P^1$}\label{geometry-diophantine}

Let $f$ be a homogeneous polynomial of degree $d$ in $n$ variables over a field $K$ and let $X\subset\PP^{n-1}$ be its projective zero locus. Assume that $X$ is smooth and let $\mathcal T_X$ be its tangent bundle. In this section we investigate the geometry of $\rho$-free rational curves  $c: \mathbb P^1 \to X$, in the sense of Definition \ref{def-free}.
It turns out that there is  a natural characterization of the $(-1)$-free curves, which we recall here.

\begin{lemma}\label{free-examples} 
A rational curve
$c: \mathbb P^1 \to X$ of degree $e$  is $(-1)$-free if and only if, in a neighborhood of $c$, the moduli space of rational curves on $X$ is smooth of dimension $(n-d)e+n-5 $.\end{lemma}

Under the assumptions of Theorem \ref{t:BV} or \cite{RY}, $\mathcal M_{0,0}(X,e)$ has dimension $(n-d)e+n-5$, so this is simply equivalent to $\mathcal M_{0,0}(X,e)$ being smooth at $c$. 

\begin{proof}[Proof of Lemma \ref{free-examples}] The tangent space of the moduli space of rational curves at $c$ is 
$$H^0(\mathbb P^1, c^* \mathcal T_X)/H^0(\mathbb P^1, \mathcal T_{\mathbb P^1}) .$$
Note that 
$H^0(\mathbb P^1, \mathcal T_{\mathbb P^1})$ has dimension $3$.   By Riemann--Roch,
\begin{align*}
  \dim H^0(\mathbb P^1, c^* \mathcal  T_X)  - \dim H^1( \mathbb P^1, c^* \mathcal T_X) &=
  \dim (c^* \mathcal T_X) + \deg (c^* \mathcal T_X) \\ &= n-2  + e (n-d) .
\end{align*}
 Hence if $c$ is a smooth point on a component of dimension $n-5 + e(n-d)$ then $H^0(\mathbb P^1, c^* \mathcal T_X)$ has dimension  $n-2 + e(n-d) $ and so $H^1( \mathbb P^1, c^* \mathcal T_X)$ vanishes. Thus  \cite[Remark 4.6]{Debarre} implies that  $c$ is ($-1$)-free. 

Conversely if $c$ is $(-1)$-free then $H^1(\mathbb P^1, c^* \mathcal T_X)$ vanishes by \cite[Remark 4.6]{Debarre}, so deformations are unobstructed. Thus  the moduli space is smooth at $c$, and the dimension of the tangent space to the moduli space is $n-5 + e(n-d) $.
\end{proof}

Let $\hat{\mathcal T}_X$  be the inverse image of $\mathcal T_X \subseteq \mathcal T_{\mathbb P^{n-1}}$ under the map $\mathcal O_{\mathbb P^{n-1}}(1)^n \to \mathcal T_{\mathbb P^{n-1}}$ in the Euler sequence \eqref{eq:euler}.  
This yields  
$$
0\to \mathcal{O}_X \to \hat{\mathcal{T}}_X \to \mathcal{T}_X\to 0,
$$
so that in particular $\hat{\mathcal{T}}_X $ is a vector bundle of rank $n-1$ on $X$. 
With this in mind, we refine Definition \ref{def-free} as follows.

\begin{defi} We say that $c: \mathbb P^1 \to X$ is strongly $\rho$-free if $c^* \hat{\mathcal T}_X \otimes \mathcal O_{\mathbb P^1}(-\rho)$ is globally generated.
\end{defi} 

We thank Paul Nelson for asking a question that suggested the above definition, and which turns out to simplify our argument compared to studying the tangent bundle directly.

\begin{lemma}\label{stronger-than} If $c$ is strongly $\rho$-free, then it is $\rho$-free. \end{lemma}

\begin{proof} This follows from the fact that $\mathcal T_X$ is a quotient of $\hat{\mathcal T}_X$ and if a vector bundle is globally generated then every quotient is globally generated. \end{proof}

\begin{lemma}\label{dimension-free} We have \[ \dim H^0 ( \mathbb P^1, c^* \hat{\mathcal T}_X \otimes \mathcal O_{\mathbb P^1}(-1-\rho) ) \geq e (n-d) - \rho (n-1) \] with equality if and only if $c$ is strongly $\rho$-free.   \end{lemma}

\begin{proof} Because $\mathcal T_X$ is the kernel of the map $df: \mathcal T_{\mathbb P^{n-1}} \to \mathcal O_{\mathbb P^{n-1}}(d)$, $\hat{\mathcal T}_X$ is the kernel of a map $\mathcal O_{\mathbb P^{n-1}}(1)^n \to \mathcal O_{\mathbb P^{n-1}}(d)$ and hence has degree $n-d$. Thus $c^ * \hat{\mathcal T}_X$ has degree $e(n-d)$. Because it has rank $n-1$, its tensor product with  with $\mathcal O_{\mathbb P^1}(-1-\rho)$ has degree 
$e(n-d) - \rho (n-1) - (n-1)$.
Hence by Riemann--Roch, the dimension of its space of global sections is 
\begin{align*}
\dim H^0 ( \mathbb P^1, c^* \hat{\mathcal T}_X &\otimes \mathcal O_{\mathbb P^1}(-1-\rho) )
\\
&=
e (n-d) - \rho (n-1)+ \dim H^1( \mathbb P^1, c^* \hat{\mathcal T}_X \otimes \mathcal O_{\mathbb P^1}(-1-\rho) ).\end{align*}

It now suffices to show that $  H^1( \mathbb P^1, c^* \hat{\mathcal T}_X \otimes \mathcal O_{\mathbb P^1}(-1-\rho) $ vanishes if and only if $c^* \hat{\mathcal T}_X \otimes \mathcal O_{\mathbb P^1}(-\rho)$ is globally generated.
We can assume that  \[ c^*\hat{ \mathcal{T}}_X= \bigoplus_{i=1}^{n-1} \mathcal O_{\mathbb P^1}(k_i).\] Then $H^1( \mathbb P^1, c^* \hat{\mathcal T}_X \otimes \mathcal O_{\mathbb P^1}(-1-\rho) )=0$ if and only if $k_i-1-\rho \geq -1$ for all $i$, which happens if and only if $k_i -\rho \geq 0$ for all $i$, which occurs if and only if $c^*\hat{\mathcal{T}}_X \otimes \mathcal O_{\mathbb P^{1}}(-\rho)$ is globally generated. \end{proof}

Vector notation such as $\bfg$ or $\bfh$ will denote $n$-tuples of polynomials in $T$.
 Let $\bfg$ be an $n$-tuple of polynomials in $T$ of degree at most $e$, at least one of degree $e$, with no common zero, and such that $f(\bfg)=0$. These conditions ensure that $(g_1: \dots : g_n)$ defines a degree $e$ map $c: \mathbb P^1 \to X$. 
 
 \begin{lemma}\label{polynomial-expression} 
 $ H^0(\mathbb P^1, c^*\hat{ \mathcal T}_X \otimes \mathcal O_{\mathbb P^1}(-1-\rho)) $ is isomorphic to the space of $n$- tuples $\bfh$ of polynomials in $T$ of degree $\leq e-1-\rho$, such that $\nabla f(\bfg) \cdot \bfh = 0$. \end{lemma}
 
 \begin{proof} 
 In this proof it will be convenient to set 
 $\mathcal{B}=c^*\hat{ \mathcal T}_X \otimes \mathcal O_{\mathbb P^1}(-1-\rho)$.
 We have an exact sequence $0 \to \hat{\mathcal T}_X  \to \mathcal O_{\mathbb P^{n-1}}(1)^n \to \mathcal O_{\mathbb P^{n-1}}(d) \to 0$, with the last map given by multiplication by the gradient of $f$. Thus we obtain  an exact sequence 
 \[0 \to \mathcal{B} \to c^* \mathcal O_{\mathbb P^{n-1}}(1)^n \otimes \mathcal O_{\mathbb P^1}(-1-\rho) \to c^* \mathcal O_{\mathbb P^{n-1}}(d)  \otimes \mathcal O_{\mathbb P^1}(-1-\rho) \to 0 \] which simplifies to 
 \[
 0 \to 
 \mathcal{B} \to \mathcal O_{\mathbb P^1}(e-1-\rho)^n \to \mathcal O_{\mathbb P^1}(de-1-\rho)  \to 0 ,\]
 because $c$ has degree $e$. 
 Applying the cohomology long exact sequence, we see that  
 $ H^0(\mathbb P^1, \mathcal{B}) $ is the kernel of the natural map 
 $$
 H^0(\mathbb P^1,  \mathcal O_{\mathbb P^1}(e-1-\rho) ^n) \to H^0( \mathbb P^1,  \mathcal O_{\mathbb P^1}(de-1-\rho)),$$ 
 given by multiplication by the gradient of $f$.
 Since $H^0(\mathbb P^1,  \mathcal O_{\mathbb P^1}(e-1-\rho) ^n)=H^0(\mathbb P^1,  \mathcal O_{\mathbb P^1}(e-1-\rho) )^n$ is the space of $n$-tuples of polynomials of degree at most $e-1-\rho$, this is exactly the stated space. 
 \end{proof}
 
 We now assume $K= \mathbb F_q$ is a finite field. Thus  $f\in \FF_q[x_1,\dots,x_n]$  is a non-singular form of degree $d\geq 3$. We assume throughout that $\mathrm{char}(\FF_q)>d$.

\begin{defi}\label{def:1}Let $N(q,e,f)$ be the number of tuples of $n$ polynomials $g_1,\dots,g_n$ over $\mathbb F_q$, of degree at most $e$, at least one of degree exactly $e$, with no common zero, such that $f(g_1,\dots,g_n)=0$.\end{defi}

\begin{defi}\label{def:Nrho} For each integer $\rho$, let $N_\rho(q,e,f)$ be the number of pairs of a tuple of polynomials $g_1,\dots,g_n$ over $\mathbb F_q$, of degree at most $e$, at least one of degree exactly $e$, with no common zero and a tuple of polynomials $h_1,\dots,h_n$ over $\mathbb F_q$, of degree at most $e-1 - \rho$, such that \eqref{eq:system} holds. 
\end{defi}

\begin{prop}\label{unfree-bound} \begin{enumerate}
\item The number of $\mathbb F_q$-points on $\mathcal M_{0,0}(X,e)$ is \[ \frac{N(q,e,f)}{(q-1) 
(q^{3}-q)} .\] 

\item The number of $\mathbb F_q$-points on $Z_{\rho}$ is at most \[  \frac{ N_\rho(q,e,f)  q^{\rho (n-1)   - e(n-d) } - N(q,e,f)}{(q-1)^2(q^3-q) } . \]

\end{enumerate}\end{prop}

\begin{proof} Each point of $\mathcal M_{0,0}(X,e)$ corresponds to $|\mathrm{PGL}_2(\mathbb F_q)| = q^3-q$ distinct maps $\mathbb P^1 \to X$. Thus in (1) we will count the number of maps $\mathbb P^1 \to X$, and in (2) we will count the number of maps $\mathbb P^1 \to X$ that are not $\rho$-free, and  in each case then divide by $q^3-q$. 

For (1), it is sufficient to note that for any such tuple $\bfg$, $(g_1: \dots : g_n)$ are the projective coordinates of a degree $e$ map $\mathbb P^1 \to X$. All such maps arise this way, and two tuples define the same map if and only if one is the multiple of the other by a non-zero scalar. 

For (2), it follows from  Lemma \ref{stronger-than} that it suffices to consider the space of degree $e$ maps  $c:\PP^1\to X$ that are not strongly $\rho$-free. 
Note that $N_{\rho}(q,e,f)$ is the sum over tuples of polynomials $(g_1, \dots , g_n)$, defining maps $c$, of $q$ raised to the dimension of the vector space of possible $h_1,\dots,h_n$. By Lemma \ref{polynomial-expression} this exponent is 
$$\dim  H^0(\mathbb P^1, c^* \hat{\mathcal T}_X \otimes \mathcal O_{\mathbb P^1}(-1-\rho)).
$$ 
By Lemma \ref{dimension-free},   $q$ to the power of this dimension is equal to $q^{  e(n-d)-  \rho (n-1) }$ if $c$ is strongly $\rho$-free and is at least $q^{  e(n-d) - \rho(n-1) +1}$ otherwise. Hence \begin{align*}
N_\rho(q,e,f)  q^{ \rho (n-1)   - e(n-d)}
&\geq 
\sum_{\substack{
\g\in \FF_q[T]^n\\
|\g|=e\\
\text{$c$ not strongly $\rho$-free}}} 
\hspace{-0.3cm}
q+
\sum_{\substack{
\g\in \FF_q[T]^n\\
|\g|=e\\
\text{$c$ strongly $\rho$-free}}} 
\hspace{-0.3cm} 1
\\
&=N(q,e,f)+(q-1)
\sum_{\substack{
\g\in \FF_q[T]^n\\
|\g|=e\\
\text{$c$ not strongly $\rho$-free}}} 
\hspace{-0.3cm}
1.
\end{align*}
The proposition  follows on noting  that there  are $(q-1)$ tuples $\g$ for each map $c:\PP^1\to X$.
\end{proof}

\section{The circle method: identification of major arcs}\label{diophantine-exponential}

For $e\geq 1$
we have 
$$
N(q,e,f)=\#\{\g\in \FF_q[T]^n:   |\g|=q^e, ~f(\g)=0, ~\gcd(g_1,\dots,g_n)=1\},
$$
where $\g=(g_1,\dots,g_n)$ and $|\g|=\max_{1\leq i\leq n}|g_i|$.
In particular only non-zero vectors $\g$ occur. 
Similarly, we may write
$$
N_\rho(q,e,f)=
\sum_{\substack{\g\in \FF_q[T]^n\\    |\g|=q^e\\  f(\g)=0\\ \gcd(g_1,\dots,g_n)=1}}
 \sum_{\substack{\h\in \FF_q[T]^n\\
|\h|< q^{e-\rho}\\ 
\h.\nabla f(\g)=0 }}1,
$$
where once again we note that only non-zero vectors $\g$ occur.
We may use
the function field analogue of the 
M\"obius function
$\mu:\FF_q[T]\to \{0,\pm 1\}$ to detect the coprimality condition
$\gcd(g_1,\dots,g_n)=1$. This gives
\begin{align*}
N_\rho(q,e,f)
&=
\sum_{\substack{k\in \FF_q[T]\\ 
\text{$k$ monic}}} \mu(k) 
\sum_{\substack{
\g\in \FF_q[T]^n\\  
 0<|\g|=q^e/|k|\\
 f(\g)=0}}
\sum_{\substack{\h\in \FF_q[T]^n\\
|\h|< q^{e-\rho}\\ 
\h.\nabla f(\g)=0 }}1\\
&=
\sum_{j\geq 0}
\sum_{\substack{k\in \FF_q[T]\\ 
|k|=q^j\\
\text{$k$ monic}}} \mu(k) 
\sum_{\substack{
\g\in \FF_q[T]^n\\  
 0<|\g|=q^{e-j}\\
 f(\g)=0}}
\sum_{\substack{\h\in \FF_q[T]^n\\
|\h|< q^{e-\rho}\\ 
\h.\nabla f(\g)=0 }}1.
\end{align*}
In view of the elementary identity
\begin{equation}
\label{eq:elementary}
\sum_{\substack{k\in \FF_q[T]\\ 
|k|=q^j\\
\text{$k$ monic}}} \mu(k) =\begin{cases}
1 & \text{ if $j=0$,}\\
-q & \text{ if $j=1$,}\\
0 & \text{ if $j>1$,}
\end{cases}
\end{equation}
it readily follows that 
$$
N_\rho(q,e,f)
=
\sum_{j\geq 0}c_j N(e-j+1,e-\rho),
$$
where
\begin{equation}\label{eq:cj}
c_j
=\begin{cases}
1 & \text{ if $j=0$,}\\
-(q+1) & \text{ if $j=1$,}\\
q & \text{ if $j=2$,}\\
0 & \text{ if $j>2$}
\end{cases}
\end{equation}
and 
$$
N(u,v)=
\sum_{\substack{
\g\in \FF_q[T]^n\\  
0< |\g|<q^{u}\\
 f(\g)=0}}
\sum_{\substack{\h\in \FF_q[T]^n\\
|\h|< q^{v}\\ 
\h.\nabla f(\g)=0 }}1,
$$
for any integers $u,v\geq 1$.

We have 
$$
\sum_{\substack{\h\in \FF_q[T]^n\\
|\h|< q^{e-\rho}\\ 
\h.\nabla f(\g)=0 }}1
=
\int_{\TT} S(\beta)\d \beta,
$$
where
$$
S(\beta)=
\sum_{\substack{\h\in \FF_q[T]^n\\  |\h|< q^{e-\rho}}}
\psi(\beta\h.\nabla f(\g)).
$$
Here the integral is over the space $\TT$ of formal Laurent series in $T^{-1}$ of degree less than $0$, against the Haar measure with total mass $1$, and $\psi$ is the additive character of $\mathbb F_q((T^{-1}))$ that sends a formal Laurent series in $T^{-1}$ to a fixed non-trivial additive character of $\mathbb F_q$ applied to the coefficient of $T^{-1}$.
With this notation we now have 
\begin{equation}\label{eq:full-integral}
N_\rho(q,e,f)
=
\sum_{j\geq 0}c_j
\sum_{\substack{
\g\in \FF_q[T]^n\\  
0< |\g|<q^{e-j+1}\\
 f(\g)=0}}
 \int_{\TT} S(\beta)\d \beta.
\end{equation}
Our plan will be to define a set of major arcs whose total contribution to 
$q^{\rho(n-1)-e(n-d)}N_\rho(q,e,f)$ is matched by $N(q,e,f).$
We note that the sum over $\g$ is empty unless $e\geq j$, so we will be able to assume this whenever dealing with this sum. 

In what follows  we shall frequently make use of the  basic orthogonality property
\begin{equation}\label{eq:orthog}
 \sum_{\substack{b\in \FF_q[T]\\ |b|<q^B}} \psi(\gamma b) =\begin{cases}
 q^B & \text{ if $\|\gamma\|<q^{-B}$,}\\
 0 & \text{ otherwise,}
 \end{cases}
\end{equation}
which is valid for any integer $B\geq 0$ and any $\gamma\in \FF_q((T^{-1})).$
Here
we recall that 
 $\|\gamma\|=|\sum_{i\leq -1}b_iT^i|$ for any $\gamma=\sum_{i\leq N}b_i T^i \in \FF_q((T^{-1}))$.

Let $\g\in \FF_q[T]^n$ be a non-zero vector such that $f(\g)=0$.
The next result is the first step  towards defining the relevant set of major arcs for our problem.

\begin{lemma}\label{lem:1}
Suppose that $\beta=a/r+\theta$ for coprime polynomials $a,r\in \FF_q[T]$ such that $|a|<|r|\leq q^{e-\rho}$.
Assume that $|r\theta|<q^{-(d-1)(e-j)}$.
Then 
$$
S(\beta)=\begin{cases}
q^{n(e-\rho)} &\text{ if $r\mid \gcd(g_1,\dots,g_n)^{d-1}$ and $|\theta|<q^{\rho-e}/|\g|^{d-1}$,}\\
0 &\text{ otherwise}.
\end{cases}
$$
\end{lemma}

\begin{proof}
We break the sum into residue classes modulo $r$, by writing 
$\h=\u+r\mathbf{v}$ for $|\u|<|r|$ and $|\mathbf{v}|<q^{e-\rho}/|r|$.
Then 
$$
S(\beta)=
\sum_{\substack{\u\in \FF_q[T]^n\\  |\u|< |r|}}
\psi(\beta\u.\nabla f(\g)) 
\sum_{\substack{\mathbf{v}\in \FF_q[T]^n\\  |\mathbf{v}|< q^{e-\rho}/|r|}}
\psi(r\theta \mathbf{v}.\nabla f(\g)) 
$$
Since 
$|r\theta|<q^{-(d-1)(e-j)}$ we have 
$|r\theta \nabla f(\g)|\leq |r\theta| q^{(d-1)(e-j)}<1$. Thus 
$\|r\theta \nabla f(\g)\| =|r\theta \nabla f(\g)|$ and
it follows from \eqref{eq:orthog} that 
$$
\sum_{\substack{\mathbf{v}\in \FF_q[T]^n\\  |\mathbf{v}|< q^{e-\rho}/|r|}}
\psi(r\theta \mathbf{v}.\nabla f(\g)) =
\begin{cases}
|r|^{-n}q^{n(e-\rho)} & \text{ if $| \theta \nabla f(\g)|<q^{\rho-e}$,}\\
0 & \text{ otherwise}.
\end{cases}
$$
We claim that 
$|\nabla f (\g)|=|\g|^{d-1}$.
To see this suppose that 
$|\g|=q^{m}$ for a non-negative integer $m$ and let 
 $\g^*\in \FF_q^n$ be the (non-zero) leading coefficient of $\g$.  
 In particular $f(\g^*)=0$ since $f(\g)=0$.
 Since  $f$ has degree $d$  it follows that 
 the coefficient of $T^{m(d-1)}$ in 
$\nabla f(\g)$ is 
 $\nabla f(\g^*)\neq \0$, since 
 $f$ is non-singular.
 
Our argument so far shows that 
$$
S(\beta)=
\begin{cases}
|r|^{-n}q^{n(e-\rho)}T(\beta) & \text{ if $| \theta |<q^{\rho-e}/|\g|^{d-1}$,}\\
0 & \text{ otherwise},
\end{cases}
$$
where
$$
T(\beta)=
\sum_{\substack{\u\in \FF_q[T]^n\\  |\u|< |r|}}
\psi(\beta\u.\nabla f(\g)).
$$
When $| \theta |<q^{\rho-e}/|\g|^{d-1}$ it follows that 
$$
|\theta \u.\nabla f(\g)| \leq q^{-1}|\theta r \nabla f(\g)|
\leq 
q^{-2+\rho-e} |r| \leq q^{-2},
$$
since $|r|\leq q^{e-\rho}$. Hence, since $a$ and $r$ are coprime, we deduce that
$$T(\beta)=
\sum_{\substack{\u\in \FF_q[T]^n\\  |\u|< |r|}}
\psi\left(\frac{a\u.\nabla f(\g)}{r}\right) =
\begin{cases}
|r|^n &\text{ if $r\mid \nabla f(\g)$,}\\
0 & \text{ otherwise},
\end{cases}
$$
Since $f$ is a non-singular form, the statement of the lemma follows on noting that 
$r\mid \nabla f(\g)$ if and only if $r\mid \gcd(g_1,\dots,g_n)^{d-1}$.
\end{proof}

\begin{lemma}\label{lem:disjoint'} Suppose  that $e\geq \rho$ and \[ \frac{a_1}{r_1} + \theta_1 = \frac{a_2}{r_2} + \theta_2 ,\] with $r_1,r_2\mid \gcd(g_1,\dots,g_n)^{d-1}$ and $|\theta_1|, |\theta_2|<q^{\rho-e}/|\g|^{d-1}.$
 Then in fact $\frac{a_1}{r_1} = \frac{a_2}{r_2}$ (and so $\theta_1=\theta_2$).
\end{lemma} 

\begin{proof} By clearing denominators, we may assume $r_1 = r_2 = \gcd(g_1,\dots,g_n)^{d-1}$. Then 
$a_1 -a_2 =  \gcd(g_1,\dots,g_n)^{d-1} (\theta_2-\theta_1)$, so that 
$$
|a_1-a_2|  < q^{ \rho-e}  \frac{   \gcd(g_1,\dots,g_n)^{d-1} }{ |\g|^{d-1}} \leq q^{\rho-e} \leq 1. 
$$
This implies that  $a_1=a_2$, as required. 
\end{proof} 

We take as major arcs the union
\begin{equation}\label{eq:major-beta}
\mathfrak{N}_j=\bigcup_{\substack{r\in \FF_q[T] \text{ monic}\\ 
|r|\leq q^{e-\rho}
}}
\bigcup_{\substack{|a|<|r|\\ \gcd(a,r)=1}} \left\{\beta\in \FF_q((T^{-1})): \left|r\beta-a\right|<q^{-(d-1)(e-j)}\right\},
\end{equation}
for $j\geq 0$.
It follows from Lemma \ref{lem:1} that $S(\beta)$ is non-zero for $\beta\in\mathfrak{N}_j$ if and only if there is some pair $(a/r,\theta)$ such that $\beta = a/r+\theta$ and all the conditions 
$$
|r| \leq q^{e-\rho}, \quad  |\theta|< |r|^{-1}q^{-(d-1)(e-j)}, \quad  |a|<|r|, \quad \gcd(a,r)=1,$$
and 
$$ 
 r\mid \gcd(g_1,\dots,g_n)^{d-1},  \quad |\theta|<
q^{\rho-e}/|\g|^{d-1}
$$ 
are satisfied.  By Lemma \ref{lem:disjoint'},  pairs  satisfying these conditions
(or even the last three conditions) 
are unique. Hence we can rewrite the integral over the major arcs as 
\begin{align*}
\int_{\mathfrak{N}_j}S(\beta)\d \beta
&=
q^{n(e-\rho)} 
\hspace{-0.3cm}
\sum_{\substack{|r|\leq q^{e-\rho}\\ \text{$r$ monic}\\ r\mid \gcd(g_1,\dots,g_n)^{d-1}}} 
\sum_{\substack{ |a| < |r| \\ \gcd(a,r)=1 }}
\hspace{-0.3cm}
\int_{|\theta|<\min\{ 
q^{\rho-e}/|\g|^{d-1}, ~
|r|^{-1}q^{-(d-1)(e-j)}\}} \d \theta \\ 
&=
q^{n(e-\rho)} 
\hspace{-0.3cm}
\sum_{\substack{|r|\leq q^{e-\rho}\\ \text{$r$ monic}\\ r\mid \gcd(g_1,\dots,g_n)^{d-1}}} 
\hspace{-0.3cm}
\phi(r)
\int_{|\theta|<\min\{ 
q^{\rho-e}/|\g|^{d-1}, ~
|r|^{-1}q^{-(d-1)(e-j)}\}} \d \theta,
\end{align*}
for any non-zero vector $\g\in \FF_q[T]^n$ such that $f(\g)=0$,
where
$\phi(r)$ is the function field analogue of the Euler totient function.
We want to replace the integral over $\theta$ by 
$$
\int_{|\theta|<
q^{\rho-e}/|\g|^{d-1}} \d \theta=\frac{q^{\rho-e}}{|\g|^{d-1}}.
$$
The  error in doing this is at most this volume multiplied by the indicator function for the inequality
$$
|r|^{-1}q^{-(d-1)(e-j)}< q^{\rho-e}/|\g|^{d-1}.
$$
Since $r\mid \gcd(g_1,\dots,g_n)^{d-1}$   this inequality implies that 
\begin{equation}\label{eq:error-1}
q^{j+D+1}|\g|\leq |\gcd(g_1,\dots,g_n)|q^{e},
\end{equation}
where 
\begin{equation}\label{eq:D}
D=\left\lfloor\frac{e-\rho}{d-1}\right\rfloor.
\end{equation}
At this point we observe that 
$$
\sum_{\substack{\text{$r\in \FF_q[T]$ monic}\\ r\mid \gcd(g_1,\dots,g_n)^{d-1}}} 
\hspace{-0.3cm}
\phi(r)=|\gcd(g_1,\dots,g_n)|^{d-1},
$$
since $\g\neq \0$.
Note that when 
$r\mid \gcd(g_1,\dots,g_n)^{d-1}$ and $|r|>q^{e-\rho}$ we must have 
\begin{equation}\label{eq:error-2}
|\gcd(g_1,\dots,g_n)|\geq  q^{D+1},
\end{equation}
with $D$ as above.
Putting everything together it follows that
\begin{equation}\label{eq:major-estimate}
\int_{\mathfrak{N}_j}S(\beta)\d \beta
=
\frac{q^{(n-1)(e-\rho)} |\gcd(g_1,\dots,g_n)|^{d-1}}{|\g|^{d-1}}\left(1+\epsilon_j \mathbf{1}_j(\g)
\right)
\end{equation}
for $\epsilon_j\in [-1,1]$, 
where 
$$
\mathbf{1}_j(\g)=\begin{cases}
1 &\text{ if \eqref{eq:error-1} or \eqref{eq:error-2} hold,}\\
0 & \text{ otherwise}.
\end{cases}
$$

Let $N_{\rho}^{\text{major}}(q,e,f)$ denote the contribution 
to the right hand side 
of \eqref{eq:full-integral} from \eqref{eq:major-estimate} for each $j$.
We now see that 
\begin{align*}
N_{\rho}^{\text{major}}(q,e,f) &=
q^{(n-1)(e-\rho)}
\sum_{j\geq 0}c_j
\hspace{-0.3cm}
\sum_{\substack{
\g\in \FF_q[T]^n\\  
 0<|\g|<q^{e-j+1}\\
 f(\g)=0}}
 \frac{|\gcd(g_1,\dots,g_n)|^{d-1}}{|\g|^{d-1}}
 \left(1+\epsilon_j \mathbf{1}_j(\g)
\right).
\end{align*}
On noting that 
$(n-1)(e-\rho)-e(d-1)=e(n-d)-\rho(n-1)$,
the main term is seen to be 
\begin{align*}
q^{e(n-d)-\rho(n-1)}\left(
\tilde N(e)
-
q^{d} \tilde N(e-1)
\right),
\end{align*}
where for $u\geq 0$ we set
\begin{align*}\tilde N(u)
=\sum_{\substack{
\g\in \FF_q[T]^n\\  
 |\g|=q^{u}\\
 f(\g)=0}}
|\gcd(g_1,\dots,g_n)|^{d-1}
&=
\sum_{\substack{k\in \FF_q[T]\\
|k|\leq q^u\\
\text{$k$ monic}
}} |k|^{d-1}
N(q,u-\deg(k),f)\\
&=
\sum_{\ell=0}^u q^{d\ell }N(q,u-\ell,f),
\end{align*}
in the notation of Definition \ref{def:1}.
Hence
\begin{align*}
\tilde N(e)
-
q^{d} \tilde N(e-1)
	&=
\sum_{\ell=0}^e q^{d\ell }N(q,e-\ell,f)
-
\sum_{\ell=0}^{e-1} q^{d(\ell+1) }N(q,e-1-\ell,f)\\
&=N(q,e,f).
\end{align*}
\begin{remark} The cancellation here is not miraculous. The terms corresponding to $\g$ with $|\g|< q^e$ or $|\gcd(g_1,\dots,g_n)|>1$ disappear precisely because $c_j$ were the coefficients defined 
in \eqref{eq:cj}
to sieve out these terms in the first place. 
\end{remark} 

Turning to the error term we can combine \eqref{eq:error-1} and \eqref{eq:error-2}
to deduce that $\gcd(g_1,\dots,g_n)\geq q^{D+1}\min(1,q^{j-e}|\g|)=q^{D+1+j-e}|\g|$ whenever 
$\mathbf{1}_j(\g)=1$.
Hence
$$
N_{\rho}^{\text{major}}(q,e,f)-
q^{e(n-d)-\rho(n-1)}N(q,e,f)\leq  
q^{(n-1)(e-\rho)}
\sum_{j\geq 0}|c_j|
E_j,
$$
where
\begin{align*}
E_j
&=
\sum_{0\leq u\leq e-j} 
\sum_{\substack{k\in \FF_q[T] \text{ monic}\\
|k|\geq q^{D+1+j-e+u}}} \frac{|k|^{d-1}}{q^{u(d-1)}}
\#\left\{
\g\in \FF_q[T]^n: 
\begin{array}{l}
 |\g|=q^u, ~ f(\g)=0\\
k=\gcd(g_1\dots,g_n)
\end{array}
\right\}\\
&=
\sum_{0\leq u\leq e-j} 
\sum_{
\ell \geq D+1+j-e+u} \frac{q^\ell}{q^{(u-\ell)(d-1)}}
N(q,u-\ell,f).
\end{align*}
Invoking  \cite[Lemma 2.8]{BV},  we deduce that 
$
N(q,e,f)=O_f(q^{(e+1)(n-1)})
$ 
for any $n\geq 3$, where the implied constant depends at most on $f$. Hence,
since we may clearly assume that $n>d+1$, 
it follows that
\begin{equation}
\begin{split}
\label{eq:revisit}
E_j
&\ll_f
\sum_{0\leq u\leq e-j} q^{u(n-d)+n-1}
\sum_{
\ell \geq D+1+j-e+u} q^{-\ell(n-d-1)}\\
&\ll_f
q^{-(D+1)(n-d-1)}\sum_{0\leq u\leq e-j} \frac{q^{u(n-d)+n-1}}
{q^{(u-e+j)(n-d-1)}}\\
&\ll_f 
q^{
(e-j)(n-d)+n-1
-(D+1)(n-d-1)}.
\end{split}
\end{equation}
The implied constant in this estimate depends only on $f$ and not on $q$. 
Thus
$$
q^{(n-1)(e-\rho)}
\sum_{j\geq 0}|c_j|
E_j\ll_f
q^{2e(n-d)-\rho(n-1)+
de-e+n-1
-(D+1)(n-d-1)}.
$$
Putting everything together,  we may conclude as follows. 

\begin{lemma}\label{lem:pebble}
 Let $\rho\in \ZZ$ and assume that $e\geq \rho$. Then 
$$
N_{\rho}^{\text{major}}(q,e,f) 
=
q^{e(n-d)-\rho(n-1)} \left(N(q,e,f) +O_f(q^{(e+1)(n-1)-(D+1)(n-d-1)})\right),
$$
where
$D$ is given by \eqref{eq:D}.
\end{lemma}

\section{The circle method:  minor arcs}\label{s:minor}

It remains to study the quantity
\begin{equation}
\label{eq:minor-j}
N_{\rho}^{\text{minor}}(q,e,f) 
=
\sum_{j\geq 0}c_j
\sum_{\substack{
\g\in \FF_q[T]^n\\  
 0<|\g|<q^{e-j+1}\\
 f(\g)=0}}
 \int_{\mathfrak{n}_j} S(\beta) \d \beta,
\end{equation}
where $\mathfrak{n}_j$ is the complement 
in $\TT=\{\beta\in \FF_q((T^{-1})): |\beta|<1\}$
of the major arcs $\mathfrak{N}_j$ that we defined in \eqref{eq:major-beta}.
Indeed, in view of Proposition \ref{unfree-bound}(2), 
the following result is now a direct consequence of \eqref{eq:full-integral} and Lemma \ref{lem:pebble}.

\begin{lemma}
\label{lem:savepoint}
Assume that $d\geq 3$ and 
$e\geq \rho$.
Then 
$$
\#Z_\rho(\FF_q)\leq 
\frac{q^{\rho (n-1)   - e(n-d)} }{(q-1)^2(q^3-q)}
N_{\rho}^{\text{minor}}(q,e,f) 
+O_f(q^{(e+1)(n-1)-5-(D+1)(n-d-1)}),
$$
where $D$ is given by \eqref{eq:D}.
\end{lemma}

We have 
\begin{equation}\label{eq:task}
\sum_{\substack{
\g\in \FF_q[T]^n\\  
 0<|\g|<q^{e-j+1}\\
 f(\g)=0 }}
 \int_{\mathfrak{n}_j} S(\beta) \d \beta
 =
\int_{\TT} \int_{\mathfrak{n}_j} 
(S(\alpha,\beta) - q^{n (e-\rho)}  )
\d \alpha\d \beta,
\end{equation}
where
\begin{equation}\label{def:Sab}
S(\alpha,\beta)=
\sum_{\substack{
\g\in \FF_q[T]^n\\  
 |\g|<q^{e-j+1}}}\sum_{\substack{\h\in \FF_q[T]^n\\  |\h|< q^{e-\rho}}}
\psi(\alpha f(\g)+\beta\h.\nabla f(\g)).
\end{equation}
Viewed as polynomials in the $2n$ variables $(\g,\h)$ the pair of polynomials 
$f(\g)$ and $\h.\nabla f(\g)$ are homogeneous of degree $d$. 
The obvious thing to do at this point is to apply Weyl differencing $d-1$ times in the spirit of Birch. This requires one to work with a simultaneous Diophantine approximation of $\alpha$ and $\beta$, which is somewhat
wasteful. It  bears fruit provided that 
$$2n-\dim V^*>3(d-1)2^d, 
$$
where $V^*$ is the (affine)  ``Birch singular locus''. In this setting $V^*$ is the locus of $(\g,\h)\in \AA^{2n}$ such that the pair of vectors $(\nabla f(\g), \0)$ and 
$(\h.\nabla^2 f(\g), \nabla f(\g))$ are proportional.
 Since $f$ is non-singular, it follows that $V^*$ is the set of $(\g,\h)\in \AA^{2n}$ such that $\g=\0$, so that $\dim V^*=n$. In this way we see that the standard approach would require 
$n>3(d-1)2^{d}$  variables overall, although there are additional difficulties associated to having lopsided boxes. 
In our work we shall exploit the special shape of our polynomials in such a way that our estimates are only sensitive to the Diophantine approximation properties of $\alpha$ or $\beta$ independently. 
This allows us to handle half the number of variables when dealing with the sum $S(\alpha,\beta).$

In what follows it will be convenient to define the monomials
$$
P_0(T)=T^{e-j}, \quad 
P(T)=T^{e-j+1} \quad  \text{ and }  \quad Q(T)=T^{e-\rho}.
$$
Let
\begin{equation}\label{eq:MM-J}
\mathfrak{M}(J)=\bigcup_{\substack{r\in \FF_q[T] \text{ monic}\\ 
|r|\leq q^J
}}
\bigcup_{\substack{|a|<|r|\\ \gcd(a,r)=1}} \left\{\alpha\in \FF_q((T^{-1})): \left|r\alpha-a\right|<q^J|P_0|^{-d}\right\},
\end{equation}
for any integer $J$.
Note that $\mathfrak{M}(-1)=\emptyset$. Let 
\begin{equation}\label{eq:MM}
M=\left\lceil \frac{d(e-j)}{2} \right\rceil.
\end{equation}
According to the function field version 
of Dirichlet's approximation theorem 
any element of  $\TT$ has a representation 
$a/r+\theta$ with $|a|<|r|\leq q^{M}$ and $|r\theta|<q^{-M}$. Hence we can cover $\TT$ by a union of arcs 
$\mathfrak{M}(J+1)\setminus \mathfrak{M}(J)$ for integers $J$ such that 
$-1\leq J\leq M-1$.

Next, let
\begin{equation}\label{eq:NN-K}
\mathfrak{N}(K)=\bigcup_{\substack{r\in \FF_q[T] \text{ monic}\\ 
|r|\leq q^K
}}
\bigcup_{\substack{|a|<|r|\\ \gcd(a,r)=1}} \left\{\beta\in \FF_q((T^{-1})): \left|r\beta-a\right|<\frac{q^K}{|P_0|^{d-1}|Q|}\right\},
\end{equation}
for any  integer  $K$. 
We note that $\mathfrak{N}(e-\rho)=\mathfrak{N}_j$, in the notation of \eqref{eq:major-beta}.
Let
\begin{equation}\label{eq:NN}
N=\left\lceil \frac{(e-j)(d-1)+e-\rho}{2} \right\rceil.
\end{equation}
It now follows from Dirichlet's approximation theorem that
the minor arcs 
$\mathfrak{n}_j$ can be covered by 
the  union of arcs 
$\mathfrak{N}(K+1)\setminus \mathfrak{N}(K)$ for integers $K$ such that 
$e-\rho\leq K\leq N-1$. 

Observe in particular that if any minor arcs exist then $e-\rho < N$ so \begin{equation}\label{eq:minor-arcs-exist}(d-1) (e-j) >  e-\rho.\end{equation} We may thus assume \eqref{eq:minor-arcs-exist} when dealing with the minor arcs. Keeping the assumptions $d \geq 3$ and $e \geq \rho$, we see in particular that \[|P|, |Q| \geq 1.\]

Our plan is  to produce two estimates for $S(\alpha,\beta)$: one for  when 
$\alpha$ belongs to $ \mathfrak{M}(J+1)\setminus \mathfrak{M}(J)$  and one for when
$\beta$ belongs to $ \mathfrak{N}(K+1)\setminus \mathfrak{N}(K)$.
Before proceeding further we note that 
\begin{equation}\label{eq:measure}
\meas\left( \mathfrak{M}(J)\right) \leq q^{2J}|P_0|^{-d}
\end{equation}
and 
\begin{equation}\label{eq:measure-N}
\meas\left( \mathfrak{N}(K)\right) \leq 
q^{2K}|P_0|^{-d+1}|Q|^{-1},
\end{equation}
for any integers $J,K\geq 0$. 

Suppose that 
$$
f(\x)
=\sum_{i_1,\dots,i_d=1}^n c_{i_1,\dots,i_d} 
x_{i_1}\dots x_{i_d},
$$
with symmetric coefficients $c_{i_1,\dots,i_d} \in \FF_q$.
Associated to $f$ are the multilinear forms
\begin{equation}\label{eq:multi}
\Psi_i(\x^{(1)},\dots,\x^{({d-1})})
=
d! 
\sum_{i_1,\dots,i_{d-1}=1}^n c_{i_1,\dots,i_{d-1},i} 
x_{i_1}^{(1)}\dots x_{i_{d-1}}^{(d-1)},
\end{equation}
for $1\leq i\leq n$.
Our first estimate for $S(\alpha,\beta)$ involves 
summing trivially over $\h$ and then applying  
Weyl differencing
$d-1$ times to the  sum over $\g$.  This eliminates the effect of the lower degree term $\beta \h.\nabla f(\g)$ and leads one to a family of linear exponential sums 
with phase vectors 
$(\alpha\Psi_1(\underline{\g}), \dots,\alpha\Psi_n(\underline\g))$, for 
 $\underline\g=(\g_1,\dots,\g_{d-1})\in \FF_q[T]^{(d-1)n}$.
This approach closely parallels 
 \cite{BV'}. 
 
 An alternative estimate for $S(\alpha,\beta)$ is obtained by applying Weyl differencing $d-2$ times to the sum over $\g$. After a further application of Cauchy--Schwarz one then brings the $\h$-sum inside, giving a  family of linear exponential sums with 
phase vectors $(\beta\Psi_1(\underline{\g}), \dots,\beta\Psi_n(\underline\g))$, for 
 $\underline\g\in \FF_q[T]^{(d-1)n}$. This brings the Diophantine properties of $\beta$ into play but extra difficulties arise from the fact that $P$ and $Q$ need not have the same degree.

\subsection{Geometry-of-numbers redux}

We shall need to begin by revisiting  a function field lattice point counting result that played a key role in \cite{BV'}.  
A
 lattice in $\FF_q((T^{-1}))^N$ is a set of points of the form ${\bf x} = \Lambda \bf u$ where $\Lambda$ is an $N \times N$ invertible matrix over $\FF_q((T^{-1}))$ and $ {\bf u}$ runs over elements of $\FF_q[T]^n$. Given a lattice $\Lambda$, the adjoint lattice is defined as the lattice associated to the inverse transpose matrix $\Lambda^{-T}$.
 
 \begin{remark} We can view lattices as vector bundles on $\mathbb P^1$ by viewing the matrix $\Lambda$ as giving gluing data for gluing the trivial vector bundle on $\mathbb A^1$ and the trivial vector bundle on a formal neighborhood of $\infty$, using the Beauville--Laszlo theorem. The adjoint lattice corresponds to the dual vector bundle, and the geometry-of-numbers computations in this section  could instead be stated  in this language. 
\end{remark}

Bearing our notation in mind we recall a
version of the ``shrinking lemma'' that is proved in \cite[Lemma 6.4]{sawin}.

\begin{lemma}\label{new-geometry} 
Let $\gamma$ be a symmetric $n \times n$ matrix with entries in $\FF_q((T^{-1}))$. 
Let $a, c,s\in \mathbb{Z}$ such that $c >0$ and $s \geq 0$.
Let $N_{\gamma,a,c}$ be the number of ${\bf x} \in \FF_q[T]^n$ such that $| {\bf x} | < q^a$ and $\| \gamma {\bf x} \| < q^{-c}$.  Then
\[ 
\frac{N_{\gamma,a,c}}{N_{\gamma,a-s,c+s}} \leq q^{ns + n  \max( \lfloor \frac{a-c}{2} \rfloor, 0 )}. \] \end{lemma}

For any 
$\alpha\in \FF_q((T^{-1}))$ and any $r> 0$, we set
\begin{equation}\label{eq:Nar}
N(\alpha;r)=
\#\left\{ \underline{\g}\in \FF_q[T]^{(d-1)n}: 
\begin{array}{l}
|\g_1|,\dots,|\g_{d-1}|<|P|\\
\|\alpha \Psi_i(\underline{\g}) \|<q^{-r} ~(\forall   i\leq n)
\end{array}{}
\right\}.
\end{equation}
Furthermore, for an integer $s \geq 0$, we put
$$
N_{s}(\alpha;r)=
\#\left\{ \underline{\g}\in \FF_q[T]^{(d-1)n}: 
\begin{array}{l}
|\g_1|,\dots,|\g_{d-1}|<|P|/q^{s}\\
\|\alpha \Psi_i(\underline{\g}) \|<q^{-r-(d-1) s}  ~(\forall   i\leq n)
\end{array}{}
\right\}.$$
We can use the shrinking lemma to bound the ratio of these two quantities as follows.

\begin{lemma}\label{lem:shrink}

  For $r>0$ and 
 ${s}\geq \max(0, e-j+1-r)$, 
we have 
\[\frac{N(\alpha,r)}{N_{s}(\alpha,r)} \leq 
q^{ (d-1) n s + n \max(0,\lfloor \frac{ e-j+1-r}{2}\rfloor )} .
\]  

\end{lemma}
\begin{proof}
For each $v\in \{0,\dots,d-1\}$, 
let  $N^{(v)}(\alpha,r)$  be the number of vectors $\underline{\g}\in \FF_q[T]^{(d-1)n}$
such that 
\begin{equation}\label{eq:lapel}
|\g_1|,\dots,|\g_{v}|<|P|/q^{s}, \quad 
|\g_{v+1}|,\dots,|\g_{d-1}|<|P| 
\end{equation}
and 
$
\|\alpha \Psi_i(\underline{\g}) \|<q^{-r- v },
$
for $1\leq i\leq n$.
Thus we have $N^{(0)}(\alpha,r)=N(\alpha,r)$ and $N^{(d-1)}(\alpha,r)=N_{s}(\alpha,r)$. 

Fix a choice of $v\in \{1,\dots,d-1\}$ and let  $\g_1,\dots,\g_{v-1},\g_{v+1},\dots,\g_{d-1}\in \FF_q[T]^n$ such that \eqref{eq:lapel} holds.
 We consider the linear forms
\begin{align*}
L_i(\g)
&=\alpha \Psi_i(\g_1,\dots,\g_{v-1},\g,\g_{v+1},\dots,\g_{d-1}),
\end{align*}
for $1\leq i\leq n$.
These form an $n \times n$ matrix. Because $\Psi_i$ is the dualization in one variable of a symmetric $d$-linear form, this $n\times n$ matrix is symmetric. The contribution to $N^{(v-1)}(\alpha,r)$ from tuples with the chosen $\g_1,\dots,\g_{v-1},\g_{v+1},\dots,\g_{d-1}\in \FF_q[T]^n$  is  $N_{\gamma, e-j+1, r + (v-1) s }$ while the contribution to 
$N^{(v)}(\alpha,r)$
from tuples of the same form is $N_{\gamma, e-j+1-s , r + v s }$. Note that $r + (v-1) s\geq r>0$ for $v\geq 1$ and so  Lemma~\ref{new-geometry} is applicable. We 
deduce that
\[ 
\frac{ N^{(v-1)}(\alpha,r)}{ N^{(v)}(\alpha,r)}  \leq   q^{ns + n   \max( \lfloor \frac{e - j+1-r  - (v-1) s }{2} \rfloor, 0 )}
\]
for $1\leq v\leq d-1$.

We take the product of this inequality over all $v$ from $1$ to $d-1$. The first term 
in the exponent 
contributes $(d-1) ns$. The second contributes $n\max( \lfloor \frac{ e -j+1-r }{2} \rfloor, 0)$ for $v=1$ and $0$ for all other values of $v$, on assuming  that $s\geq e-j+1-r$. Thus we get the  stated bound. \end{proof}

\subsection{Weyl differencing}

Our fundamental tool for estimating 
$S(\alpha,\beta)$ is Weyl differencing.  
We recall first that $|P|,|Q|\geq 1$ in this exponential sum. 
Appealing to \cite[Eq.~(5.2)]{BV'} first, Weyl differencing $d-1$ times gives
$$
|S(\alpha,\beta)|\leq 
|P|^n|Q|^n\left(
|P|^{-(d-1)n}N(\alpha, e-j+1)\right)^{1/2^{d-1}},
$$
in the notation of \eqref{eq:Nar}.  Note that as $N(\alpha,e-j+1) \geq 1$ and $2^{d-1} \geq (d-1) $, the right side is $\geq |Q|^n$. 
Thus we have
\begin{equation}\label{eq:weyl1'}
|S(\alpha,\beta) - q^{n(e-\rho)}  |\leq 2
|P|^n|Q|^n\left(
|P|^{-(d-1)n}N(\alpha, e-j+1)\right)^{1/2^{d-1}},
\end{equation}

We can also obtain an upper bound for $S(\alpha,\beta)$ that only uses information about 
$\beta$.
Let us put 
$$
T(\h)=\sum_{|\g|<|P|} 
\psi(\alpha f(\g)+\beta\h.\nabla f(\g)),
$$
so that 
$$
S(\alpha,\beta)=
\sum_{ |\h|< |Q|} T(\h),
$$
with $P,Q$ are as before.
It follows from Cauchy--Schwarz that
\begin{equation}\label{eq:S-cauchy}
|S(\alpha,\beta)|^{2^{d-2}}\leq 
|Q|^{(2^{d-2}-1)n}
\sum_{ |\h|< |Q|} |T(\h)|^{2^{d-2}}.
\end{equation}
After $d-3$ applications of Weyl differencing we obtain
$$
|T(\h)|^{2^{d-3}} \leq |P|^{(2^{d-3}-d+2)n} \sum_{\g_1,\dots,\g_{d-3}}\left|
\sum_{\g} \psi\left(D(\g)\right)\right|,
$$
where $D(\g)=D_{\g_1,\dots,\g_{d-3}}(\alpha f(\g)+\beta\h.\nabla f(\g))$ and 
$D_{\g_1,\dots,\g_{d-3}}$ is the usual differencing operator. Here $\g_1,\dots,\g_{d-3},\g$ each run over vectors in $\FF_q[T]^n$ formed from polynomials of degree less than $e-j+1$.
A further application of Cauchy--Schwarz now yields
$$
|T(\h)|^{2^{d-2}} \leq |P|^{(2^{d-2}-d+1)n} \sum_{\g_1,\dots,\g_{d-3}}\left|
\sum_{\g} \psi\left(D(\g)\right)\right|^2.
$$
Differencing once more therefore  leads to the expression
$$
\left|\sum_{\g} \psi\left(D(\g)\right)\right|^2=
\sum_{\g_{d-2},\g_{d-1}} \psi\left(D_{\g_1,\dots,\g_{d-2}}(\alpha f(\g_{d-1})+\beta\h.\nabla f(\g_{d-1})\right), 
$$
where
$$
D_{\g_1,\dots,\g_{d-2}}(\h.\nabla f(\g_{d-1})=\sum_{i=1}^n  h_i\Psi_i(\g_1,\dots,\g_{d-1}),
$$
in the notation of \eqref{eq:multi}.  
Returning to \eqref{eq:S-cauchy} 
we ignore the Diophantine approximation properties of $\alpha$ and instead execute the linear exponential  sum over $\h$. This leads to the expression
$$ 
|S(\alpha,\beta)|\leq |P|^n|Q|^n\left(|P|^{-(d-1)n}
 N(\beta,e-\rho)\right)^{1/2^{d-2}},
$$
in the notation of \eqref{eq:Nar}. Again, $N(\beta,e-\rho) \geq 1$ and $2^{d-2} \geq (d-1)$ so the right side is $\geq |Q|^n$, whence
\begin{equation}\label{eq:weyl2'}
|S(\alpha,\beta) - q^{n (e-\rho)} |\leq 
2 
|P|^n|Q|^n\left(|P|^{-(d-1)n}
 N(\beta,e-\rho)\right)^{1/2^{d-2}}.
\end{equation}

\begin{remark} \label{:P}
When $\mathrm{char}(\FF_q)\leq d$ the polynomials $\Psi_i$  are identically zero for $1\leq i\leq n$, so that 
\eqref{eq:weyl1'} and 
\eqref{eq:weyl2'} give  nothing beyond the trivial bound for the exponential sum $S(\alpha,\beta)$.
\end{remark}

Recall the definitions \eqref{eq:MM-J} and 
\eqref{eq:NN-K} of 
$\mathfrak{M}(J)$ and 
$\mathfrak{N}(K)$, respectively. 
We want to bound  the size of $S(\alpha,\beta)$ when $\alpha\not\in \mathfrak{M}(J)$
and  $\beta\not\in \mathfrak{N}(K)$. To do this it will be convenient to  introduce two parameters 
$s_1$ and $s_2$.
Associated to these are the quantities 
$$
l_1=e-j+1-s_1 \quad \text{ and } \quad
l_2=e-j+1-s_2.
$$
We can use our geometry-of-numbers shrinking result to  establishing the following pair of estimates.

\begin{lemma}\label{lem:N1} 
Let $\alpha\not\in \mathfrak{M}(J)$
and let 
$l_1\in \ZZ$ be such that 
$$
l_1\leq 1+\frac{J}{d-1} \quad \text{ and }
\quad 
l_1\leq e-j+1.
$$
Then 
there exists a constant $c_{d,n}>0$ such that. 
$$
N(\alpha, e-j+1)\leq 
c_{d,n} 
q^{-nl_1}|P|^{(d-1)n}.
$$
\end{lemma}

\begin{proof}
It follows from Lemma \ref{lem:shrink} that 
$N(\alpha,e-j+1)$ is at most 
$$
q^{(d-1)ns_1}
\#\left\{ \underline{\g}\in \FF_q[T]^{(d-1)n}: 
\begin{array}{l}
|\g_1|,\dots,|\g_{d-1}|<|P|/q^{s_1}\\
\|\alpha \Psi_i(\underline{\g}) \|<|P|^{-1}q^{-(d-1) s_1}  ~(\forall   i\leq n)
\end{array}{}
\right\},
$$
for any $s_1\geq 0$.
Note that  $|P| /q ^{s_1} = q^{l_1}$ and $q^{s_1} = |P_0| / q^{l_1 -1} $. 
Suppose that \[|\g_1|,\dots,|\g_{d-1}|<|P|/q^{s_1}\]
and 
$\|\alpha \Psi_i(\underline{\g}) \|<|P|^{-1}q^{-(d-1) s_1}$  
but $\Psi_i(\underline{\g})\neq 0$.   Let $r = \Psi_i(\underline{\g}) $ and let $a$ be the integer part of $\alpha \Psi_i(\underline{\g})$, each divided through by any common factors that they might share. 
Then $|r| \leq q^{(d-1) (l_1-1)}$ and $$
| r \alpha -a | <  |P|^{-1}  q^{ - (d-1) s_1}  = q^{  (d-1) (l_1-1)-1}   |P|_0^{-d} .
$$ 
This contradicts the assumption that   $\alpha \not \in \mathfrak{M}(J)$, 
if $J \geq  (d-1)(  l_1-1)$.
Hence,  if $ J \geq (d-1) (l_1-1)$ and $\alpha\not \in  \mathfrak M(J)$, 
we have 
$$
N(\alpha,e-j+1)\leq q^{(d-1)ns_1}
\#\left\{ \underline{\g}\in \FF_q[T]^{(d-1)n}: 
\begin{array}{l}
|\g_1|,\dots,|\g_{d-1}|<q^{l_1}\\
 \Psi_i(\underline{\g})=0~(\forall   i\leq n)
\end{array}{}
\right\}.
$$
The statement of the lemma follows on noting that 
the remaining cardinality is  $O(   q^{ (d-2) n  l_1})$ for dimensionality reasons,
where the implied constant depends only on $d$ and $n$. 
\end{proof}

\begin{lemma}\label{lem:N2}
Let $\beta\not\in \mathfrak{N}(K)$
and let  $l_2\in \ZZ$  be such that 
$$
l_2\leq 1+\frac{K}{d-1}\quad
\text{ and } \quad
l_2\leq e-j+1-\max(0,\rho-j+1).
$$
Then 
there exists a constant $c_{d,n}>0$ such that. 
$$
N(\beta, e-\rho)\leq 
c_{d,n}
q^{-nl_2 +
n\max(0,\lfloor \frac{\rho-j+1}{2}\rfloor)}
|P|^{(d-1)n}.
$$
\end{lemma}

\begin{proof}
This time we take  $r=e-\rho$ in  Lemma \ref{lem:shrink} and deduce that 
\begin{align*}
N(\beta,e-\rho)&\leq~
q^{(d-1)ns_2+n\max(0,\lfloor \frac{\rho-j+1}{2}\rfloor)}\\
&\times
\#\left\{ \underline{\g}\in \FF_q[T]^{(d-1)n}: 
\begin{array}{l}
|\g_1|,\dots,|\g_{d-1}|<q^{l_2}\\
\|\beta \Psi_i(\underline{\g}) \|<|Q|^{-1}q^{-(d-1) s_2}  ~(\forall   i\leq n)
\end{array}{}
\right\},
\end{align*} 
for any $s_2\geq \max(0,\rho-j+1)$.
Arguing as in the previous result it is simple to check that we must in fact have 
$\Psi_i(\underline{\g})=0$ for all $1\leq i\leq n$ whenever 
$\beta\not \in \mathfrak{N}(K)$ and $K\geq (d-1)(l_2-1)$. 
But then there are 
$O(   q^{ (d-2) n  l_2})$ possible vectors 
$\underline{\g}\in \FF_q[T]^{(d-1)n}$ that contribute. The statement of the lemma follows.
\end{proof}

In our work we shall take 
\begin{equation}\label{eq:donut}
l_1=
1+\left\lfloor\frac{J}{d-1}\right \rfloor, \quad
l_2=
1+\left\lfloor\frac{K}{d-1}\right \rfloor.
\end{equation}
We need to check that the remaining conditions on $l_1$ and $l_2$  are satisfied in Lemmas \ref{lem:N1} and \ref{lem:N2}. 
To begin with we note that 
$$
J\leq 
\left\lceil \frac{d(e-j)}{2} \right\rceil-1\leq 
\frac{d(e-j)}{2} -\frac{1}{2}
$$
Hence for 
Lemma \ref{lem:N1} to be applicable it suffices to have 
$$
d(e-j) -1
 \leq 2(d-1)(e-j).
$$
But this is equivalent to $0\leq 1+(d-2)(e-j)$ which follows from \eqref{eq:minor-arcs-exist}. 
Next, we note that 
\[K\leq \left\lceil \frac{(e-j)(d-1)+e-\rho}{2} \right\rceil-1  \leq \frac{(e-j)(d-1)+e-\rho}{2} -\frac{1}{2},\] 
so that Lemma \ref{lem:N2} is applicable if 
$$
e-\rho-1\leq (d-1)(e-j-2\max(0,\rho-j+1)).
$$
Thus it suffices to have
\begin{equation}\label{eq:e2}
e-\rho-1\leq (d-1)(e-j)
\end{equation}
and 
\begin{equation}\label{eq:e3}
e-\rho-1\leq (d-1)(e+j-2\rho-2).
\end{equation}
However, \eqref{eq:e2} follows from \eqref{eq:minor-arcs-exist}, so it suffices to assume that  \eqref{eq:e3} holds. 

Inserting
  Lemmas \ref{lem:N1} and \ref{lem:N2} into our  Weyl differencing bounds 
\eqref{eq:weyl1'} and \eqref{eq:weyl2'}, we deduce that  there exists a constant $c_{d,n}>0$ such that
\begin{align*}
|S(\alpha,\beta) - q^{n(e-\rho)} |
&\leq c_{d,n} |P|^n|Q|^n\min\left(
q^{-nl_1/2^{d-1}}, 
q^{   - n(l_2-
n \max(0, \lfloor \frac{\rho-j+1}{2} \rfloor ))/2^{d-2}}
\right)\\
&= c_{d,n} |P|^n|Q|^n/\max\left(
q^{l_1}, 
q^{ 2l_2-
2\max(0, \lfloor \frac{\rho-j+1}{2} \rfloor)} 
\right)^{n/2^{d-1}},
\end{align*}
whenever $(\alpha,\beta)\in \mathfrak{M}(J+1)\setminus \mathfrak{M}(J)\times  
\mathfrak{N}(K+1)\setminus \mathfrak{N}(K)$ and \eqref{eq:e3} holds.
We shall proceed under the assumption that the parameter $l_2$ satisfies
\begin{equation}\label{eq:lK}
l_2-\max\left(0, \left\lfloor \frac{\rho-j+1}{2} \right\rfloor\right)\geq 0.
\end{equation}
This is precisely the circumstance under which our $\beta$-treatment is non-trivial.
Assume that $n>(d-1)2^d$, so that
$2^{d}(d-1)/n<1$. If \eqref{eq:lK} holds we can  invoke the inequality
$\max(A,B)\geq A^{2^{d}(d-1)/n}B^{1-2^{d}(d-1)/n}$, which is valid for any $A,B\geq 1.$ 
Thus it follows that 
\begin{align*}
|S(\alpha,\beta) - q^{n (e-\rho)} |
&\leq c_{d,n} \frac{q^{2(d-1)(2l_2-l_1)
-4(d-1)\max(0, \lfloor \frac{\rho-j+1}{2} \rfloor)}
|P|^n|Q|^n}{
q^{ (l_2-\max(0,\lfloor \frac{\rho-j+1}{2}\rfloor))n/2^{d-2}}}.
\end{align*}

Returning to \eqref{eq:task} we see that
\[\sum_{\substack{
\g\in \FF_q[T]^n\\  
 0<|\g|<q^{e-j+1}\\
 f(\g)=0}}
 \int_{\mathfrak{n}_j} S(\beta) \d \beta
 \leq 
 \sum_{J=-1}^{M-1}
 \sum_{K=e-\rho}^{N-1}
E(J,K)\]
where we recall from  \eqref{eq:MM} and  \eqref{eq:NN} that
$$
M=\left\lceil \frac{d(e-j)}{2} \right\rceil, \quad 
N=\left\lceil \frac{(e-j)(d-1)+e-\rho}{2} \right\rceil,
$$
and 
\[
E(J,K)
=\int_{\mathfrak{M}(J+1)\setminus \mathfrak{M}(J)} \int_{
\mathfrak{N}(K+1)\setminus \mathfrak{N}(K)} 
|S(\alpha,\beta) - q^{n (e-\rho)} |
\d \alpha\d \beta.\]
The measure of all $(\alpha,\beta)$ in the integral is at most $q^{4+2J+2K}|P_0|^{-2d+1}|Q|^{-1}$,
by \eqref{eq:measure} and \eqref{eq:measure-N}.
Let us consider the total contribution 
 \[E_{l_1, l_2} = \sum_{ J= \max( (d-1)  (l_1-1), -1 ) }^{(d-1) l_1 - 1}  \sum_{K = \max ( (d-1) (l_2-1),  e-\rho)}^{ (d-1) l_2 - 1} E_{J,K},\]  
from $J,K$ associated to integers $l_1\geq 0$ and $l_2\geq 1$ via \eqref{eq:donut}. 
Then  
\begin{align*}
E_{l_1, l_2} 
&\ll   \frac{q^{  6(d-1) l_2 } |P|^n |Q|^{n-1}|P_0|^{-2d+1}  
q^{-4(d-1)\max(0, \lfloor \frac{\rho-j+1}{2} \rfloor)}}{
q^{ (l_2-\max(0,\lfloor \frac{\rho-j+1}{2}\rfloor))n/2^{d-2}}}\\
&=q^{\Delta_j-l_2(n/2^{d-2}-6(d-1))+
\max(0,\lfloor \frac{\rho-j+1}{2}\rfloor)\left(n/2^{d-2}-4(d-1)\right)
}, 
\end{align*} 
where we have put
$$
\Delta_j=(e-j)(n-2d+1)+(e-\rho)(n-1)+n.
$$
Because $K \geq e-\rho$, we have 
$$
l_2\geq 1 + \left\lfloor \frac{e-\rho }{d-1} \right\rfloor.
$$
In particular our condition \eqref{eq:lK} is satisfied when 
\begin{equation}\label{eq:e5}
1 + \left\lfloor \frac{e-\rho }{d-1} \right\rfloor\geq 
\max\left(0, \left\lfloor \frac{\rho-j+1}{2} \right\rfloor\right).
\end{equation}
Furthermore, assuming $n > 3(d-1)   2^{d-1}$, the bound is decreasing in $l_2$, so the dominant contribution occurs when $$ l_2 = 1 + \left\lfloor \frac{e-\rho }{d-1} \right\rfloor.$$

Since there are $O(e)$ choices for $l_1$, 
our work has therefore shown that 
\begin{align*}
\sum_{\substack{
\g\in \FF_q[T]^n\\  
0< |\g|<q^{e-j+1}\\
 f(\g)=0}}
 \int_{\mathfrak{n}_j} S(\beta) \d \beta
&\ll e q^{\Delta_j-\Gamma_j},
\end{align*}
where
\begin{align*}
\Gamma_j
=~&
\left(\frac{n}{2^{d-2}}-6(d-1)\right)\left(1+
\left\lfloor \frac{e-\rho }{d-1} \right\rfloor\right) \\
&-\left(
\frac{n}{2^{d-2}}
-
4(d-1)\right)
\max\left(0,\left\lfloor \frac{\rho-j+1}{2}\right\rfloor\right)\\
=~&
\left(\frac{n}{2^{d-2}}-6(d-1)\right)\left(1+
\left\lfloor \frac{e-\rho }{d-1} \right\rfloor-
\max\left(0,\left\lfloor \frac{\rho-j+1}{2}\right\rfloor\right)
\right)\\
& -2(d-1)\max\left(0,\left\lfloor \frac{\rho-j+1}{2}\right\rfloor\right).
\end{align*}
Thus we certainly require \eqref{eq:e5} to hold in order to expect any saving in our minor arc estimate. 

We summarise our argument in the following result. 

\begin{lemma}\label{lem:minor}
Let  $d\geq 3$ and $n>3(d-1)2^{d-1}$.
Assume that $\rho\geq -1$  and 
\begin{equation}\label{eq:final-e}
e\geq 
\max\left( \rho + (d-1)  \left\lfloor \frac{\rho+1}{2}\right\rfloor,  (\rho+1)\left(2+ \frac{1}{d-2} \right) \right).
\end{equation}
Then
$$
N_{\rho}^{\text{minor}}(q,e,f) 
\ll e
 q^{\Delta_0-\Gamma_0}
$$
where
$$
\Delta_0=2e(n-d)-\rho(n-1)+n
$$
and 
$$
\Gamma_0=
\left(\frac{n}{2^{d-2}}-6(d-1)\right)\left(1+
\left\lfloor \frac{e-\rho }{d-1} \right\rfloor-
\left\lfloor \frac{\rho+1}{2}\right\rfloor\right)
 -2(d-1)\left\lfloor \frac{\rho+1}{2}\right\rfloor.
$$
\end{lemma}

\begin{proof}
Recall  \eqref{eq:minor-j} and note that 
$\Delta_j=\Delta_0-j(n-2d+1)$. Hence for the range of $n$ in which we are interested  we deduce from \eqref{eq:cj}  that 
$$
|c_j|q^{\Delta_j-\Gamma_j}\ll q^{\Delta_0-\Gamma_0},
$$
for all $j\geq 0$. Moreover, $\Gamma_0$ takes the value recorded in the statement of the lemma when 
 $\rho\geq -1$ and the condition 
 \eqref{eq:final-e} on $e$ is enough to ensure that
  \eqref{eq:e3} 
  and  \eqref{eq:e5}
  both hold 
   for every $j\in \{0,1,2\}$. (For  \eqref{eq:e3} we note that  it suffices to have 
$
e(d-2)\geq (\rho+1)(2d-3).$)  
This  completes the proof.
\end{proof}

\subsection{Deduction of Theorem \ref{t:BV}}
We assume that 
 $n>(2d-1) 2^{d-1}$.  
We revisit the argument deployed in \cite{BV} to establish the irreducibility 
 and dimension of $\mathcal{M}_{0,0}(X,e)$. 
 This is based on a counting argument over a finite field $\FF_q$ whose characteristic is greater than the degree $d$ of the non-singular form $f\in \FF_q[x_1,\dots,x_n]$ that defines $X$.
 According to \cite[Eq.~(3.3)]{BV},
 in order to deduce that
 $\mathcal{M}_{0,0}(X,e)$ is irreducible and of the expected dimension it suffices to show that 
 \begin{equation}\label{eq:goat}
 \lim_{q\to \infty}  q^{-(n-d)e-n+1} \hat N(q,e,f)\leq 1,
 \end{equation}
 where $\hat N(q,e,f)$ is the number of $\g\in \FF_q[T]^n$ such that 
 $|\g|<q^{e+1}$ and $f(\g)=0$. 

 We have 
 $$
 \hat N(q,e,f)=\int_{\TT}S_{\text{BV}}(\alpha) \d\alpha,
 $$
 where
$$
S_{\text{BV}}(\alpha)
=\sum_{\substack{\g\in \FF_q[T]^n\\ |\g|<q^{e+1}}} \psi(\alpha f(\g))
 =
 q^{-n(e-\rho)} S(\alpha,0),
 $$
 in the notation of \eqref{def:Sab}, with $j=0$. 
Take $j=0$ in the major arcs  
$\mathfrak{M}(J)$ 
that were 
defined in 
\eqref{eq:MM-J}.  A straightforward calculation shows that the contribution from 
the major arc around $0$ is 
$$
\int_{\mathfrak{M}(0)}S_{\text{BV}}(\alpha) \d\alpha=
\sum_{\substack{\g\in \FF_q[T]^n\\ |\g|<q^{e+1}}} 
\int_{|\theta|<q^{-de}}
\psi(\theta f(\g)) 
 \d\theta = q^{ne-de}\left(q^{n-1}+O(q^{n/2})\right).
$$
In order to complete the proof of \eqref{eq:goat} it therefore suffices to show that 
$$
\lim_{q\to \infty} q^{-(n-d)e-n+1}\sum_{J=0}^{M-1} \int_{\mathfrak{M}(J+1)\setminus \mathfrak{M}(J)} |S_{\text{BV}}(\alpha)|\d\alpha <1
$$
where 
$M=\lceil \frac{de}{2}\rceil$
is given by \eqref{eq:MM}. 
To do this we may 
apply our previous work.
Thus it follows from
\eqref{eq:weyl1'} and 
 Lemma \ref{lem:N1}
that
$$
S_{\text{BV}}(\alpha)\ll |P|^n q^{-nl_1/2^{d-1}},
$$
if $\alpha\not\in \mathfrak{M}(J)$ and  $l_1$ is any integer such that 
$l_1\leq 1+J/(d-1)$ and 
$l_1\leq e+1$.
 The choice $l_1=1+\lfloor J/(d-1)\rfloor$
 is acceptable since $J\leq \lceil \frac{de}{2}\rceil-1\leq \frac{de-1}{2}
 $, whence
$$
l_1\leq 1+\frac{de-1}{2(d-1)
}
=1+e,
$$
for   $d\geq 3$. Since $J\geq 0$ we are clearly only interested in integers $l_1\geq 1$.  
Appealing to \eqref{eq:measure} to estimate the volume of $\mathfrak{M}(J+1)$, we deduce that for given $l_1\geq 1$ the  total 
associated 
contribution is
\begin{align*}
\sum_{J=(d-1)(l_1-1)}^{(d-1)l_1-1} \int_{\mathfrak{M}(J+1)\setminus \mathfrak{M}(J)} |S_{\text{BV}}(\alpha)|\d\alpha 
&\ll 
\sum_{J=(d-1)(l_1-1)}^{(d-1)l_1-1} q^{2J+2-de} . 
|P|^n q^{-nl_1/2^{d-1}}\\
&\ll 
q^{-de+n(e+1) +(2(d-1)-n/2^{d-1})l_1}.
\end{align*}
 This is decreasing with $l_1$
 if  $n> (d-1) 2^{d}$ 
  and we may therefore sum over $l_1\geq 1$ to finally deduce that 
 $$
 q^{-(n-d)e-n+1}\sum_{J=0}^{M-1} \int_{\mathfrak{M}(J+1)\setminus \mathfrak{M}(J)} |S_{\text{BV}}(\alpha)|\d\alpha
 \ll  q^{1+2(d-1)-n/2^{d-1}}.
$$ 
The exponent of $q$ is negative if 
 $n> (2d-1) 2^{d-1}$, which thereby concludes the proof  of \eqref{eq:goat}, whence 
 $\mathcal M_{0,0}(X,e)$ is indeed irreducible and of the expected dimension. It follows from the same method used in \cite[p.~2] {HRS} that  $\mathcal M_{0,0}(X,e)$ is locally a complete intersection. Indeed, since  $\mathcal M_{0,0}(X,e)$ is locally the intersection of $de+1$ equations in $\mathcal {M}_{0,0}(\PP^{n-1},e)$, a smooth stack of dimension $ne -4$, it is a locally complete intersection if and only if its dimension is $(n-d)e + n-5$.

\subsection{Deduction of Theorem \ref{t:1}}

Assume that $d\geq 3$, $n>3(d-1)2^{d-1}$,  $\rho\geq -1$, and $e \geq (\rho+1)\left(2+ \frac{1}{d-2} \right)$. In particular, this implies that $e \geq \rho$, which is needed for Lemma \ref{lem:savepoint}.  In view of Theorem \ref{t:BV}, the stated bound is trivial unless  $1 + \lfloor \frac{e-\rho} {d-1} \rfloor -  \left\lfloor \frac{\rho+1}{2}\right\rfloor> 0 $, so we may assume that 
$ \lfloor \frac{e-\rho} {d-1} \rfloor -  \left\lfloor \frac{\rho+1}{2}\right\rfloor \geq 0 $ 
and thus $e \geq \rho + (d-1) \left\lfloor \frac{\rho+1}{2}\right\rfloor $. Hence we may assume that \eqref{eq:final-e} holds. 
 
Combining 
Lemmas \ref{lem:savepoint} and 
\ref{lem:minor} we deduce that 
\begin{equation}\label{eq:count}
\#Z_\rho(\FF_q)\ll
 e
q^{e(n-d)+n-5-\min(\mu_1(n),\mu_2(n))},
\end{equation}
with 
\begin{align*} 
\mu_1(n)=\left(\frac{n}{2^{d-2}}-6(d-1)\right)
\left(1+D
- \left\lfloor \frac{\rho+1}{2}\right\rfloor\right)
-2(d-1)\left\lfloor \frac{\rho+1}{2}\right\rfloor
\end{align*}
and 
$$
\mu_2(n)=(1+D)(n-d-1)-de+e+1.
$$
 Here we recall that  $D$ is given by \eqref{eq:D} as $\lfloor \frac{e-\rho} {d-1} \rfloor$.  
 
 We claim that $\mu_1(n)\leq \mu_2(n)$. They are both increasing affine functions of $n$, with $\mu_1(n)$ of lesser slope than $\mu_2(n)$. Hence to check that $\mu_1(n)$ is the minimum, it suffices to check that $\mu_2(n) \geq 0$  and $\mu_1(n) \leq 0$ when $n = 3(d-1) 2^{d-1}$. In other words, we must show that
  \[ 
  3 (d-1) 2^{d-1} \geq  d+1 + \frac{ e (d-1) -1 } {1 + D }. 
  \] 
 To do this, observe that because $e \geq  \rho + (d-1) \left\lfloor \frac{\rho+1}{2}\right \rfloor$ we have $e \geq \frac{d+1}{2} \rho$,   so that 
 \[ 1 + D \geq   \frac{e +1 - \rho }{d-1}   \geq \frac{ e + 1 - \frac{2}{d+1} e } {d-1} \geq \frac{e}{d+1} .\]
  Thus 
 \[  d+1 + \frac{ e (d-1) -1 } {1 + D } \leq  d+ 1 + \frac{ e(d-1)}{e/(d+1)} = d (d+1), \]
 so it suffices to check \[ 3(d-1) 2^{d-1}  \geq d(d+1).\] 
 But it is clear that 
 this 
  holds for all $d\geq 3$, whence  $\mu_2(n) \geq  \mu_1(n)$.

By Lang--Weil \cite{LW}, it now follows from \eqref{eq:count}  that 
$$\dim Z_{\rho}\leq  e(n-d)+n-5-\mu_1(n)$$ 
for any smooth hypersurface defined over a finite field. For a general hypersurface, we can spread it out to a family defined over a ring finitely-generated over $\mathbb Z$. The dimension of $Z_\rho$ in this family is manifestly constant on some open subset of the spectrum of this ring, which must contain a finite-field valued point, so  $\dim Z_{\rho}$ is at most $e(n-d)+n-5-\mu_1(n)$ for the generic point and thus for the original hypersurface. This completes the proof of Theorem~\ref{t:1}.

\subsection{Deduction of Theorem \ref{t:starr}}

We consider the effect of taking  $\rho=-1$ in Theorem \ref{t:1}. 
Clearly \eqref{eq:e-condition} is equivalent to $e\geq 0$ and can be ignored.  
Note that $Z_{-1}$ contains the singular locus of $\mathcal{M}_{0,0}(X,e)$
by \cite[Thm.~2.6]{Debarre}. Thus the codimension of the singular locus is at least $\dim \mathcal M_{0,0}(X,e)- \dim Z_{-1}$. Theorem~\ref{t:starr} therefore follows from applying Theorem \ref{t:BV} to calculate  $\dim \mathcal M_{0,0}(X,e)$ and Theorem \ref{t:1} to bound  $\dim Z_{-1}$.

Because the lower bound for the codimension of the singular locus is strictly positive, the moduli space is generically smooth. Any generically smooth locally complete intersection scheme is reduced, which thereby completes the proof of Theorem \ref{t:BV}.

\section{Peyre's freedom counting function}\label{s:peyre}

In this section we prove the asymptotic formula in Theorem \ref{t:P}
for the counting function \eqref{eq:NX},
by piecing together our work above and 
the main results in Lee's thesis \cite{lee}.    We have 
\begin{equation}\label{free-not}
N_X^{\ve\text{-free}}(B)=N_X(B)-E_\ve(B),
\end{equation}
where $E_\ve(B)$ counts the number of $x\in X(\FF_q(T))$ with $H_{\omega_V^{-1}}(x)\leq q^B$
such that $\ell(x)<\ve$.

Let us begin by studying $N_X(B)$. As usual we suppose that $X$ is defined by a non-singular form 
$f\in \FF_q[x_1,\dots,x_n]$ of degree $d\geq 3$.
It follows from the proof of part (1) of Proposition~\ref{unfree-bound} that 
$$
N_X(B)
=\frac{1}{q-1} \#\left\{
\g\in \FF_q[T]^n: 
\begin{array}{l}
\gcd(g_1,\dots,g_n)=1\\
|\g|^{n-d}<q^{B+1}, ~f(\g)=0
\end{array}
\right\}.
$$
Using  the M\"obius function to detect the coprimality condition we obtain 
\begin{align*}
N_X(B)
&=\frac{1}{q-1}
\sum_{\substack{k\in \FF_q[T]\\ 
\text{$k$ monic}}} \mu(k) \#\left\{
\g\in \FF_q[T]^n: 0< |k \g|^{n-d}<q^{B+1}, ~ f(\g)=0\right\}\\
&=\frac{1}{q-1}
\sum_{j\geq  0}
\sum_{\substack{k\in \FF_q[T]\\ |k|=q^j\\
\text{$k$ monic}}} \mu(k) 
\#\left\{
\g\in \FF_q[T]^n:  
\begin{array}{l}
0<|\g|^{n-d}<q^{B+1-j(n-d)}\\ f(\g)=0
\end{array}
\right\}.
\end{align*}
Put
$
m=n-(d-1)2^d$ and assume that $m>0$. 
Then, on 
appealing to Lee's thesis  \cite[Thm.~4.1.1]{lee}, 
it follows that 
\begin{equation}
\label{eq:lee}
\begin{split}
\#\big\{\g\in \FF_q[T]^n: &~ 0<|\g|^{n-d}<q^{R+1}, ~f(\g)=0\big\}
\\&=
q^{R}\left(c_f +O(q^{-mR/(2^{d+1}(d-1)(n-d))})\right),
\end{split}
\end{equation}
for any $R>0$, 
where $c_f$ is the usual product of singular series and singular integral.
Using  \eqref{eq:elementary} to handle the sum over $j$ and $k$, it now follows from \eqref{eq:lee} 
that there exists $\delta>0$ such that 
\begin{align*}
N_X(B)
&=\frac{c_f}{(q-1)\zeta_{\FF_q(T)}(n-d)}
q^B
+O\left(q^{(1-\delta)B}\right),
\end{align*}
where 
$\zeta_{\FF_q(T)}(s)=(1-q^{1-s})^{-1}$ is the rational zeta function. 
Arguing along standard lines (as in Peyre \cite[\S 5.4]{peyre-duke}, for example), one readily confirms that this agrees with the Batyrev--Manin--Peyre prediction for the hypersurface $X$.

It remains to produce an upper bound for the quantity $E_\ve(B)$ in \eqref{free-not}.
Let $x\in X(\FF_q(T))$ and suppose that it defines 
a map $c:\PP^1\to X$ of degree $e$.
Then it follows from 
 \cite[Notation 5.7]{peyre-freedom} that 
$$
\ell(x)=\frac{(n-1)\rho}{e(n-d)}
$$
if and only if $c$  is $\rho$-free but not $(\rho+1)$-free. 
(In particular,  Remark \ref{rem:1.4} implies that $\ell(x)\in [0,1]$.)
We deduce that  $E_\ve(B)$ is at most the number of rational maps from $\PP^1\to X$ with degree at most $B/(n-d)$
which are not $\rho$-free, with 
\begin{equation}\label{eq:rho-eps}
\rho=\left\lfloor \frac{\ve B}{n-1}\right \rfloor +2.
\end{equation}
We may therefore appeal to the proof of Proposition~\ref{unfree-bound}(2) to estimate this quantity, finding that
$$
E_\ve(B)\leq \frac{N_\rho(q,B/(n-d),f)q^{\rho(n-1)-B} -N(q,B/(n-d),f)}{(q-1)^2},
$$
with $\rho$ given by \eqref{eq:rho-eps}.
In what follows 
it will be convenient to set $e=B/(n-d)$ and to assume that $e\in \NN$.
All of the implied constants that follow are allowed to depend on $q$ and $f$, but not on $e$ or $\rho$.
We seek conditions on $n$ and $\rho$ under which we can deduce that there exists $\delta>0$ such that 
$E_\ve(B)=O(q^{(1-\delta)e(n-d)})$.

First we improve our treatment of Lemma \ref{lem:pebble} slightly by revisiting the 
argument \eqref{eq:revisit}. Since we no longer care about a dependence on the finite field, rather than invoking a trivial bound  we may apply \eqref{eq:lee} to deduce that 
$
N(q,u-\ell,f)\ll q^{(u-\ell)(n-d)}$
if $n>(d-1)2^d$. But then \eqref{eq:revisit} can be replaced by the bound 
\begin{align*}
E_j
&\ll_f
\sum_{0\leq u\leq e-j} q^{u(n-2d+1)}
\sum_{
\ell \geq D+1+j-e+u} q^{-\ell(n-2d)}\\
&\ll_f 
q^{
(e-j)(n-2d+1)
-(D+1)(n-2d)},
\end{align*}
where $D$ is given by \eqref{eq:D},
whence
\begin{align*}
q^{(n-1)(e-\rho)}
\sum_{j\geq 0}|c_j|
E_j
&\ll_f
q^{2e(n-d)-\rho(n-1)-(D+1)(n-2d)}\\
&\ll_f 
q^{2e(n-d)-\rho(n-1)-(e-\rho)(n-2d)/(d-1)}. 
\end{align*}
It now  follows from \eqref{eq:full-integral} and our modified version of Lemma \ref{lem:pebble} that 
$$
E_\ve(B)\ll q^{e(n-d)-(e-\rho)(n-2d)/(d-1)}+
q^{-e(n-d)+\rho(n-1)}N_{\rho}^{\text{minor}}(q,e,f),
$$
provided that 
$e\geq \rho$. 
Note that $\Gamma_0=
\gamma_0+O_{d,n}(1)
$,
with 
$$
\gamma_0=\left(\frac{n}{2^{d-2}}-6(d-1)\right)\left(\frac{e-\rho}{d-1}-\frac{\rho}{2}\right)
-(d-1)\rho.
$$
Appealing now to Lemma \ref{lem:minor} we therefore  deduce that 
$$
E_\ve(B)\ll 
q^{e(n-d)-(e-\rho)(n-2d)/(d-1)}+
e q^{e(n-d) -\gamma_0}
$$
if  \eqref{eq:final-e} holds.

Recall that $n>3(d-1)2^{d-1}$. 
Then $n/2^{d-2}-6(d-1)\geq 2^{-d+2}$ and we can ensure that 
$\gamma_0\geq \delta e$ for a small parameter $\delta>0$ (that depends only on $d$)
provided that 
\begin{equation}\label{eq:final-hyp}
e\geq (d-1)^2 2^{d-1}\rho.
\end{equation}
This is also enough to ensure 
that $(e-\rho)(n-2d)/(d-1)\geq \delta e$.
This inequality is clearly much stronger than \eqref{eq:final-e}.
The statement of Theorem~\ref{t:P} now follows on taking  $e=B/(n-d)$ and noting that the hypothesis on $\ve$ in the theorem is enough to ensure that 
\eqref{eq:final-hyp} holds when $\rho$
is given by
\eqref{eq:rho-eps} and $B$ is sufficiently large.

\end{document}